\documentclass[10pt]{article}

 \usepackage{color}

\usepackage{amsmath,amscd,amsfonts,amsthm, xypic, amssymb}
\input{xy}

\newcommand{\shrinkmargins}[1]{
  \addtolength{\textheight}{#1\topmargin}
  \addtolength{\textheight}{#1\topmargin}
  \addtolength{\textwidth}{#1\oddsidemargin}
  \addtolength{\textwidth}{#1\evensidemargin}
  \addtolength{\topmargin}{-#1\topmargin}
  \addtolength{\oddsidemargin}{-#1\oddsidemargin}
  \addtolength{\evensidemargin}{-#1\evensidemargin}
  }

\shrinkmargins{.7}

\DeclareMathOperator{\Hom}{Hom}

\DeclareMathOperator{\Spf}{Spf}

\DeclareMathOperator{\Tor}{Tor}
\DeclareMathOperator{\Map}{Map}
\DeclareMathOperator{\Ext}{Ext}

\newcommand{\BP}[1]{BP \langle {#1} \rangle}
\newcommand{\field}[1]{\mathbb{#1}}

\newcommand{\Z}{\field{Z}}
\newcommand{\N}{\field{N}}
\newcommand{\F}{\field{F}}
\newcommand{\W}{\field{W}}

\newcommand{\R}{\field{R}}
\newcommand{\C}{\field{C}}
\newcommand{\G}{\field{G}}

\newcommand{\EE} {\mathcal{E}}
\newcommand{\UU} {\mathcal{U}}
\newcommand{\RR} {\mathcal{R}}

\newcommand{\id}{\mbox{id}}

\newcommand{\m}{\mathfrak{m}}

\newcommand{\tK}{\widetilde{K}}
\newcommand{\tQ}{\widetilde{Q}}
\newcommand{\tx}{\widetilde{x}}
\newcommand{\ty}{\widetilde{y}}
\newcommand{\tz}{\widetilde{z}}
\newcommand{\SqZ}{Sq_{\Z}}
\newcommand{\unk}{\underline{k}}
\newcommand{\unE}{\underline{E}}
\newcommand{\unK}{\underline{K}}

\newcommand{\beq}{\begin{displaymath}}
\newcommand{\eeq}{\end{displaymath}}
\newcommand{\beqn}{\begin{equation}}
\newcommand{\eeqn}{\end{equation}}

\theoremstyle{plain}
\newtheorem{thm}{Theorem}[section]
\newtheorem{prop}[thm]{Proposition}
\newtheorem{cor}[thm]{Corollary}
\newtheorem{lem}[thm]{Lemma}

\theoremstyle{definition}
\newtheorem{defn}[thm]{Definition}
\newtheorem{conj}[thm]{Conjecture}
\newtheorem{exmp}[thm]{Example}

\theoremstyle{remark}
\newtheorem{rem}[thm]{Remark}

\newtheorem{notation}[thm]{Notation}

\newtheorem{assump}[thm]{Assumption}

\title{Twisted Morava K-theory and E-theory}
\author{Hisham Sati and Craig Westerland}

\begin{document}

\bibliographystyle{amsalpha}

\maketitle

\begin{abstract}

For a class $H \in H^{n+2}(X; \Z)$, we define twisted Morava K-theory $K(n)^*(X; H)$ at the prime $2$, as well as an integral analogue.  We explore properties of this twisted cohomology theory, studying a twisted Atiyah-Hirzebruch spectral sequence, a universal coefficient theorem (in the spirit of Khorami).  We extend the construction to define twisted Morava E-theory, and  provide applications to string theory and M-theory.

\end{abstract}

\medskip
Examples of twisted cohomology theories have been growing in recent years.  One favorite example is the twisted K-theory $K^*(X; H)$ of a space $X$; here the twisting $H$ is usually taken to be an element $H \in H^3(X, \Z)$ (e.g., \cite{ar, abg, agg, as1, as2, BCMMS, dk, FHT, Ka, Ro, txl}).  
Periodic de Rham cohomology may be twisted by any odd degree cohomology class (see e.g. \cite{BSS, MW, teleman}), and Ando-Blumberg-Gepner \cite{abg} have recently shown that the theory $tmf$ of topological modular forms admits twistings by elements of $H^4(X, \Z)$.

Twistings of a suitably multiplicative (that is, $A_\infty$) cohomology theory $E^*$ are governed by a space $BGL_1 (E)$.  While universal in this respect, its topology is far from being transparent.  One may, however, summarize the above examples by 
saying that there exist essential, continuous maps
$$\begin{array}{cccc}
K(\Z, 3) \to BGL_1(K), & K(\R, 2n+1) \to BGL_1(H\R[u, u^{-1}]), & {\rm and} & K(\Z, 4) \to BGL_1(tmf),
\end{array}$$
where $u$ is the periodicity element
and $K(G, n)$ is an Eilenberg-MacLane space whose sole homotopy group is $G$ in dimension $n$.  Since these spaces represent the cohomology functor $H^n(-, G)$, such cohomology classes give rise to twistings of the indicated generalized cohomology theories.

One naturally begins to ask what sort of twistings other familiar cohomology theories support.  It is the purpose of this paper to investigate this question for Morava's extraordinary K-theories $K(n)$, their $2$-periodic variant, $K_n$, and E-theories $E_n$, introduced in \cite{morava}.  These cohomology theories exist for each prime number $p$ and $n \in \N$.  Their value on a point is 
$$K(n)_*=\F_{p}[v_n^{\pm 1}], \quad {K_n}_* = \F_{p^n}[ u^{\pm 1}], \quad \mbox{and } \quad {E_n}_* = \W(k)[[u_1, \dots, u_{n-1}]][ u^{\pm 1}].$$
Here, $|v_n| = 2(p^n-1)$, $|u_i| = 0$, $|u| = 2$, and $\W(k)$ is the ring of Witt vectors over a perfect field $k$ of characteristic $p$ (often taken to be $\F_{p^n}$).  Furthermore, these theories are complex-oriented, and support height $n$ formal group laws.  Constructed via homotopical methods, they do not as yet admit geometric descriptions akin to K-theory or cohomology.  Concomitantly, there is no geometry at hand to provide twisted versions of these theories (in contrast to K-theory, where the map $K(\Z, 3) \to BGL_1 K$ has a natural description in terms of tensor products of line bundles).

Nonetheless, we may approach the study of twistings of $K(n)$ (by cohomology classes) via homotopy theory.  As indicated above, we do so by studying the space of maps\footnote{This philosophy and some of the methods are resonant with the recent paper of Antieau-Gepner-G\'{o}mez \cite{agg}.} from $K(\Z, m)$ to $BGL_1 K(n)$ or $BGL_1 K_n$. We find that there are no nontrivial twistings by $H^m(X, \Z)$ when $m>n+2$.  More interestingly, we find that when $m=n+2$, there are many such twistings: 

\paragraph{Theorem 1.}
{\it For $2$-periodic Morava K-theory $K_n$, each component in the space of maps $K(\Z, n+2) \to BGL_1 K_n$ is contractible, and the set of components is canonically isomorphic to a group of algebra homomorphisms:
$$[K(\Z, n+2), BGL_1 K_n] \cong \Hom_{\textit{$K_{n*}$-alg}}({K_n}_*K(\Z, n+1), {K_n}_*).$$
Further, the latter group is isomorphic to the $p$-adic integers $\Z_p$.  The same holds for $K(n)$ at the prime $p=2$.} 

\vspace{2mm}

A representative $u:K(\Z, n+2) \to BGL_1 K_n$ of a particular topological generator of this group will be called the \emph{universal} twisting.  This allows us to define, for any space $X$ and class $H \in H^{n+2}(X)$, the \emph{twisted Morava K-theory} $K_n^*(X; H)$ (and $K(n)^*(X; H)$, when p=2).

We give two natural methods of computing these twisted cohomology groups.  First, we construct a twisted form of the Atiyah-Hirzebruch spectral sequence for $K(n)^*$ (Theorem \ref{ahss_thm}), and prove a formula for the first nontrivial differential in this spectral sequence.  Secondly, we prove a universal coefficient theorem (Theorem \ref{khorami_thm}) along the lines of Khorami's result \cite{khorami} for twisted K-theory.  This allows one to compute the twisted $K_n$-homology $K_{n *}(X; H)$ quite explicitly in terms of the \emph{untwisted} $K_n$-homology of the principal $K(\Z, n+1)$-bundle over $X$ defined by the class $H$.

While $K_n$ and $K(n)$ are only $A_\infty$-ring spectra, $E_n$ has an $E_\infty$-multiplication, so its space of units $GL_1 E_n$ deloops to a spectrum of units $gl_1 E_n$. A similar result holds in this setting:

\paragraph{Theorem 2.} \label{theorem2}
{\it For each $n\geq 1$, there are canonical isomorphisms
$$\pi_0 \Map_{E_\infty} (K(\Z, n+2), BGL_1 E_n) \cong [\Sigma^{n+1} H\Z, gl_1 E_n] \cong \Hom_{\textit{$E_{n*}$-alg}}({E_n}_*K(\Z, n+1), {E_n}_*),$$
and again the latter group is isomorphic to the $p$-adic integers $\Z_p$.}

\vspace{2mm}
Here $H \Z$ is the integral Eilenberg-MacLane spectrum.  Choosing a topological generator $\varphi_n: H\Z \to gl_1 E_n$, the map induced on infinite loop spaces allows us to twist $E_n$ by a class in $H^{n+2}$, as we do for $K_n$.  
We will also show that twisted Morava K-theory is, in some sense, 
a reduction of twisted Morava E-theory.

The proofs of Theorems 1 and 2 rely heavily on the computations of the Morava K-theory of Eilenberg-MacLane spaces, originally due to Ravenel-Wilson \cite{rw}, and recently revisited in Hopkins-Lurie \cite{ambidexterity}.  We compute the homotopy type of the space of these maps using an obstruction theory due to Robinson, Goerss-Hopkins, and others \cite{r, gh, gh2}; the Ravenel-Wilson computations yield a vanishing of the obstruction groups and an identification of the set of components in terms of a set of algebra homomorphisms.  Using \cite{ambidexterity} or the work of Buchstaber-Lazarev \cite{BL}, one may identify the $p$-divisible group associated to ${K_n}_*K(\Z, n+1)$ with the top exterior power $\Lambda^n G$ of the formal group $G$ of $K_n$.  Further, there is a non-canonical isomorphism $\Lambda^n G \cong \G_m$, from which the isomorphism of the group of twists with $\Z_p$ may be derived.  For those who are unenthusiastic about obstruction-theoretic arguments, we conjecture (see Conjecture \ref{conjecture} below) that $\varphi_n$ may be constructed directly using a splitting in the spaces constituting the Johnson-Wilson spectrum $\BP{n}$.

At the prime 2, Morava K-theory $K(n)$ is the mod 2 reduction of an ``integral\footnote{Really, 2-adic.} lift" $\tK(n)$ with coefficient 
ring $\tK(n)_*=\Z_2[v_n, v_n^{-1}]$.
It turns out that the twisted cohomology theory $K(n)^*(X; H)$ also lifts to an integral version $\tK(n)^*(X; H)$.  Since $\tK(n)$ is constructed from an appropriate choice of $E_n$ by killing certain generators, this twisting comes about as a direct consequence of the E-theory twistings.  The integral theory is also computable via an Atiyah-Hirzebruch spectral sequence.  We do not yet have an analogue of Khorami's theorem for the integral or E-theories; we expect that if such a result holds, its proof will be much harder than the mod 2 version, and more akin to Khorami's.

Nonetheless, the integral theory is much more suited to applications in physics. We conclude the paper with examples and applications motivated by physics within string theory and M-theory. 
In \cite{KS1} it was shown that an anomaly cancellation condition of the form $W_7=0$ 
can be interpreted as an orientation condition with respect to second integral Morava K-theory
at the prime 2, hence characterizing admissible spacetimes in M-theory. 
We show that in this setting, as well as in heterotic string theory, 
we get conditions of the form $W_7 + \alpha=0 \in H^7(X; \Z)$;
 we interpret this via the Atiyah-Hirzebruch spectral sequence as an orientation 
with respect to twisted second integral Morava K-theory of spacetime, thus providing 
a {\it twisted} version of the results in \cite{KS1}. We employ the full relative version
of twisted integral Morava K-theory. We will also make connection to twisted String structures
\cite{Wan} \cite{SSS3} and twisted Fivebrane structures \cite{SSS3}.

\tableofcontents

\vspace{3mm}
\noindent {\it Acknowledgements.} It will become quickly apparent that we owe a great debt to Ravenel-Wilson's remarkable computations \cite{rw} of the Morava K-theories of Eilenberg-MacLane spaces.  We would like to thank Matthew Ando, Vigleik Angeltveit, Anssi Lahtinen, Haynes Miller, Eric Peterson, Charles Rezk, and Steve Wilson for very helpful conversations on this material, and David Baraglia, Tyler Lawson, and Bryan Wang for comments on an earlier version of this article.  We also thank the referee for encouraging us to extend our original results to odd primes via $2$-periodic Morava K-theory. H. S. would like to thank the Department of Mathematics at the University of Melbourne and IHES, Bures-sur-Yvette,  
for hospitality during the work on this project.  His research is supported by 
NSF Grant PHY-1102218. C. W.  was supported by the Australian Research Council 
(ARC) under the 
Future Fellowship and Discovery Project schemes.

\section{Twistings of cohomology theories}

We will study twistings of a cohomology theory $E^*$.  For such a theory, we write $E$ for the associated omega spectrum; 
that is $E = (\unE_n)_{n \in \Z}$, where $\unE_n$ are topological space equipped with homeomorphisms $f: \unE_n \to \Omega \unE_{n+1}$.  
The cohomology theory is represented by $E$: the 
space $E^n(X) = [X, \unE_n]$ is the abelian group of homotopy classes of maps $X \to \unE_n$.

All of the cohomology theories that we will consider will have $E^*(X)$ a ring; this endows the spectrum $E$ 
with a multiplicative structure.  
In the case of complex K-theory and Morava E-theory the resulting spectrum is an $E_\infty$ ring spectrum 
(that is, homotopy commutative in a highly structured sense) while that of Morava K-theory is noncommutative, and is only an $A_\infty$-spectrum (that is, homotopy associative).

\subsection{Units of ring spectra}
\label{sec unit}

Now assume that $E$ is an $A_\infty$ ring spectrum.  The ring structure on the homotopy groups of $E$ allows one to define the 
topological space $GL_1 E$ of units in $E$ as the fiber
 product
$$\xymatrix{
GL_1 E \ar[r] \ar[d] & \Omega^\infty E = \unE_0 \ar[d]^-{\pi_0} \\
\pi_0(E)^{\times} \ar[r]_-{\subseteq} & \pi_0(E),
}$$
where $\pi_0(E)^{\times} \subseteq \pi_0(E)$ is the group of units of the ring $\pi_0(E)$.  

The $A_\infty$-structure on $E$ equips $GL_1 E$ with a grouplike $A_\infty$-multiplication.  This allows us to define a classifying space 
$BGL_1 E$.  We recall that $GL_1$ participates in an adjunction with the suspension spectrum functor (see (6.12) of \cite{abg}):
$$\pi_0 \Map_{A_\infty}(\Sigma^\infty Z_+, E) \cong \pi_0 \Map_{A_\infty}(Z, GL_1 E).$$
Here $Z$ is an $A_\infty$ space, and $\Map_{A_\infty}$ denotes the space of $A_\infty$ maps (of spaces on the right, and of spectra on the left).  This should be compared to \cite{abghr}, where an $E_\infty$ version (which we will also use in section \ref{e_section}) is proved.

It is worth pointing out that an $A_\infty$ map $f: Z \to GL_1 E$ defines a continuous map $Bf: BZ \to BGL_1E$ of 
classifying spaces.  If $Z$ is a loop space (that is, $Z \simeq \Omega BZ$), the homotopy class of $f$ may be 
recovered as the loop map $f \simeq \Omega Bf$. Indeed, there is a bijection between homotopy classes of such maps.  
Thus we may conclude:

\begin{prop} \label{maps_prop}

For a loop space $Z$, there is a natural bijection
$$\pi_0 \Map_{A_\infty}(\Sigma^\infty Z_+, E) \cong [BZ, BGL_1 E].$$

\end{prop}

\subsection{Twistings}
\label{sec twistings}

A general homotopy-theoretic description of twisted spectra can be found in \cite{abghr, MSi}.

\begin{defn}

For an $A_\infty$-ring spectrum $E$, a \emph{twisting} of $E$ by a space $Y$ will be a map $Y \to BGL_1 E$.  Two 
twistings are \emph{isomorphic} if they are homotopic.\footnote{We note that this gives the usual definition of isomorphic twists in the case of K-theory.}  We will write $tw_E(Y)$ for the set of isomorphism classes of twistings of $E$ by $Y$; that is,
$$tw_E(Y) := [Y, BGL_1E] = \pi_0(\Map(Y, BGL_1E)).$$

\end{defn}

Let $F$ denote the homotopy functor from topological spaces to sets represented by $Y$; that is, $F(X) = [X, Y]$.  Then 
any twisting $u \in tw_E(Y)$ defines a 
twisted cohomology theory (see \cite{MSi, abg}) 
$$u_*: \{(X, H), \; H \in F(X)\} \to \mbox{$E^*$-modules}.$$
We will write $E^*(X; H)$ for $u_*(X, H)$.

Let us outline the construction of $E^*(X; H)$, following \cite{abg}.  In order to minimize technology (e.g., the language of $\infty$-categories), we employ their definition using Proposition 6.14.  The composite $u \circ H: X \to BGL_1 E$ defines a principal $GL_1(E)$ bundle $p:  P_H \to X$.  The \emph{generalized Thom spectrum} is
$$X^{ H} := \Sigma^{\infty} (P_H)_+ \wedge_{\Sigma^{\infty} GL_1(E)_+} E,$$
where $E$ is made into a $\Sigma^{\infty} GL_1(E)_+$-module via the counit of the adjunction defining $GL_1$.  
Note that $X^{ H}$ is an $E$-module by action on the right.

\begin{defn} \label{twist_defn}

The \emph{$n^{\rm th}$ $H$-twisted $E$-(co)homology groups of $X$}, are defined to be the homotopy groups
$$\begin{array}{ccc}
E_n(X; H) := \pi_n(X^{ H}) & {\rm and} & E^n(X; H) = \pi_{-n}F_E(X^{ H}, E),
\end{array}$$
where $F_E(A,B)$ denotes the spectrum of $E$-module maps $A \to B$.

For a subspace $i: A \subseteq X$, there is an induced map $i^{ H}: A^{ H \circ i} \to X^{ H}$.  The cofiber $C$ of this map remains an $E$-module, and the relative twisted (co)homology groups are defined as 
$$\begin{array}{ccc}
E_n(X, A; H) := \pi_n(C) & {\rm and} & E^n(X, A; H) = \pi_{-n}F_E(C, E),
\end{array}$$

\end{defn}

\begin{rem}

An important caveat is in order.  The bundle $P_H$ depends upon the choice of representative $f_H: X \to Y$ of the homotopy class $H$.  Although different choices give isomorphic bundles, the isomorphism is non-canonical, unless $u_* F(\Sigma X) = 0$.  Thus different representatives of $H$ give rise to non-canonically isomorphic twisted cohomology groups $E_n(X; H)$.  This is a familiar phenomenon in twisted K-theory (e.g., \cite{BEM}), where the group depends upon the (deRham) representative of the twisting class $H$.

\end{rem}

We recall that in the case of twists of K-theory, this rather homotopy theoretic definition reduces to the construction of Atiyah-Segal in terms of sections of bundles whose fibers are the terms in the K-theory spectrum.

Twisted generalized cohomology resembles cohomology with local coefficients.
This leads to a fairly general framework for twisted generalized cohomology theories, which 
we now describe, following the presentation in \cite{Ku}. See also \cite{Wa}.

Let $Top{}_2$ denote the category of pairs of CW-complexes. 
For each sequence of integers and abelian groups $(n, G):=(n_i, G_i)_{1\leq i \leq k}$, we define the category $(n, G)$-$Top{}_2$ of \emph{spaces with $(n, G)$-twist}.  This is the category with objects pairs
of objects $(X, A)$ of $Top{}_2$ and a sequence $(a_i)_{1\leq i \leq k}$
of maps $a_i: X \to K(G_i, n_i-1)$, that is representatives for cohomology classes 
$H^{n_i}(X; G_i)$ and morphisms from $(X, A, (a_i))$ to 
$(Y, B, (b_i))$, with morphisms $f: (X, A) \to (Y, B)$ in $Top{}_2$ such that $f^*b_i:= b_i \circ f=a_i$.

\begin{defn}[Twisted generalized cohomology theory, after \cite{Ku}] \label{gen_defn}
An \emph{$(n, G)$-twisted cohomology theory} is a sequence 
$(E^n, \sigma^n)_{n \in \Z}$ of contravariant functors
$E^n: (n, G)$-$Top_2$ $\to$ $AbGrps$ and maps 
$\sigma^n: E^n(A, \emptyset, (i^*a_i)) \to E^{n+1}(X, A, (a_i))$, natural in 
objects $(X, A, (a_i))$, such that the following properties hold:

\begin{enumerate}
 
 \item {(\it Excision)}
 For $(X, A, (a_i))$ and $U\subset X$ such that $\overline{U} \subset int(A)$,
 if $i: (X/U, A/U, (i^*_{X/U}a_i)) \to (X, A, (a_i))$ is the inclusion then all $E^n(i)$ are isomorphisms.
 
\item ({\it Homotopy invariance}) For two homotopic maps $f$ and $g$, 
 $E^n(f)=E^n(g)$ for all $n \in \Z$. 
 
 \item ({\it Multiplicativity}) For each indexing set $J$ and set of pairs $\{(X_j, A_j, (a_i^j))\}_{j\in J}$
 in  $Top_2$, we have an isomorphism 
 $E^n(\coprod_{j\in J})(X_j, A_j, (a_i^j)) \cong \prod_{j \in J} E^n(X_j, A_j, (a_i^j))$
 given by the product of the maps induced by the inclusions. 
 
 \item ({\it Long exact sequence of a pair)}
 For $(X, A, (a_i))$, let the maps $i: A \to X$ and $j: X \to (X, A)$ be the inclusions. Then the sequence 
 $$\small{
 \xymatrix@1{
 \cdots 
 \ar[r]
 &
 E^n(X, A, (a_i))
 \ar[r]^{E^n(j)}
 &
 E^n(X, \emptyset, (a_i))
 \ar[r]^{E^n(i)}
 &
 E^n(A, \emptyset, (i^*_Aa_i))
 \ar[r]^{\sigma^n}
 &
 E^{n+1}(X, A, (a_i))
 \ar[r]
 &
 \cdots
 }}
 $$
 is exact.
 \end{enumerate}
The set $\prod_{i=1}^k H^{n_i}(-, G_i)$ is called the set of twists of the twisted generalized
cohomology theory.
\end{defn}

There is of course a similar set of axioms for twisted homology theories.  When the set of twists is empty, the above definition reduces to the usual Eilenberg-Steenrod axioms for a cohomology theory.  Definition \ref{twist_defn} is immediately seen to satisfy these axioms, as it is given by the homotopy groups of a spectrum.

\vspace{3mm}
Before discussing twists for Morava K-theory and E-theory,
 we recall twists of complex K-theory. 

\paragraph{Topological K-theory.}
Complex K-theory is represented by the $\Omega$-spectrum $\{\unK_n\}_{n \geq 0}$, where
$\unK_n=\Z \times BU$ if $n$ is even and $\unK_n=U$ if $n$ is odd. $\unK_0=\Z\times BU$ 
is an $E_\infty$-ring space for which the corresponding space of units 
$GL_1 K= K_\otimes$ is the infinite loop space is $\Z/2 \times BU_\otimes$, where 
$BU_\otimes$ denotes the space $BU$ with the $H$-space structure induced 
by the tensor product of virtual vector bundles of virtual dimension 1. 
Hence, if $X$ is a compact connected space, units in the ring 
$K^*(X)$ under tensor product are represented by virtual vector bundles of dimension 
$\pm 1$. The determinant defines a splitting 
$$
GL_1^+(K^*(X)) \cong {\rm Pic}(X) \times SL_1(K^*(X))\;,
$$
where ${\rm Pic}(X)$ is the Picard group of topological line bundles and 
$SL_1(K^*(X))$ denotes 1-dimensional virtual bundles with trivialized
determinant line. At the level of the spectrum $BU_\otimes$ of 1-dimensional 
units in the classifying spectrum for complex K-theory one has the following
factorization \cite{MST}
$$
BU_\otimes \cong K(\Z,2) \times BSU_\otimes\;.
$$
A twisting of K-theory over $X$ is a principal $\Z / 2 \times BU_\otimes$ bundle over $X$.
Altogether, twistings of K-theory over a compact space $X$ are classified by 
homotopy classes of maps $X \to K(\Z/2, 1) \times K(\Z, 3) \times BBSU_\otimes$,
and so correspond to elements in $H^1(X;\Z/2) \times H^3(X;\Z) \times 
[X, BBSU_\otimes]$.
Such an element is then
a triple $(\alpha, \beta, \gamma)$ consisting of a $\Z/2$-twist $\alpha$, 
a determinantal twist $\beta$, which is a $K(\Z,2)$ bundle over $X$, and a higher 
twist $\gamma$, which is a $BSU_\otimes$-bundle.

\vspace{3mm}
In the following sections we will explore similarities and differences between 
the above and the case of Morava K-theory and E-theory.

\subsection{Obstruction theory}
\label{sec obs}

Proposition \ref{maps_prop} implies that, for $Z = \Omega X$, there is a natural bijection
$$tw_E(X) = tw_E(BZ) \cong \pi_0 \Map_{A_\infty}(\Sigma^\infty Z_+, E).$$
Therefore, in order to explore such twistings, we must get a handle on the set of homotopy classes of 
$A_\infty$-maps $\Sigma^\infty Z_+ \to E$.  
There is an obstruction theory developed by Robinson \cite{r} to compute the topology of the space of
 such maps.  We will follow Goerss-Hopkins' 
presentation \cite{gh} of these results.

We work over a base $A_\infty$ ring spectrum $R$, and take $E$ to be an $R$-module.  We will perform 
a number of homological algebraic 
computations over the ring $R_*$ of coefficients of $R$.  In particular, if $A$ is an $R_*$-algebra, 
and $M$ an $A$-bimodule, we will 
write $HH^*(A, M)$ for the Hochschild cohomology of $A$ with coefficients in $M$ \emph{in the 
category of $R_*$-modules}. 

\begin{prop} \label{obstruction_prop}

If the groups $HH^{k}(R_* Z, \Omega^s E_*)$ vanish for $s=k-1$ and $k-2$, and every $k \geq 2$, then the 
Hurewicz map
$$tw_E(X) \cong \pi_0(\Map_{A_\infty}(\Sigma^\infty Z_+, E)) \to \Hom_{\textit{$R_*$-alg}}(R_*(Z), E_*)$$
is a bijection.

\end{prop}

Here $\Omega^s E_*$ is $E_*$ shifted down in degree by $s$.  This result is Corollary 4.4 of \cite{gh} (or rather,
 a version of it incorporating Theorem 4.5) applied to the case $\mathcal{F} = A_\infty$, along with the 
recognition (e.g. Example 2.4.5 of \cite{gh2}) that the Goerss-Hopkins obstruction groups $D^s_{R_*(Ass)}(A, M)$ 
are the Hochschild cohomology groups $HH^{s+1}(A, M)$.

A special case is given when $E=R$ and $E_*(Z)$ is flat over $E_*$.  The Hochschild cohomology groups in question 
are then
$$
HH^{k}(E_* Z, \Omega^s E_*) = \Ext^k_{E_*(Z) \otimes_{E_*} E_*(Z)^{op}}(E_*(Z), \Omega^s E_*) \cong \Ext^k_{E_*(Z)^{op}}(E_*, \Omega^s E_*).
$$
We note that since we have assumed that $E_*(Z)$ is flat over $E_*$, products of $Z$ admit a K\"unneth 
isomorphism\footnote{This is always the case when $E$ is a Morava K-theory or singular homology with field coefficients.}.  Thus 
the latter may be identified as the $E^{k, s}_2$-term of the $E_*$ cobar spectral sequence converging to the
 $E^*$-\emph{cohomology} $E^*(BZ)$: 

\begin{cor} \label{obs_cor}

If $E_*(Z)$ is flat over $E_*$, the obstructions to the existence and uniqueness of the realization of a ring 
map $E_*(Z) \to E_*$ as a map $BZ \to BGL_1(E)$ of spaces lie in the $E_2$-term of the cobar spectral sequence:
$$\Ext^k_{E_*(Z)^{op}}(E_*, \Omega^s E_*) \implies E^{k-s}(BZ), \mbox{ $k \geq 2$ and $s=k-1, \, k-2$.}$$

\end{cor}

\section{Twistings of $K(n)$-local theories} \label{twisting_section}

In this section we consider twistings of Morava $K(n)$-theory as well as of $K(n)$-local theories, 
such as Morava $E$-theory. We will work mod $p$; that is, we are 
considering the rings $K(n)_*=\F_{p}[v_n , v_n^{-1}]$, where $v_n$ is a generator 
of degree $|v_n|=2(p^n-1)$. Later in section \ref{sec int} we will consider the integral case.  Throughout, we equip $K(n)$ with its homotopically unique $A_\infty$ structure, as in \cite{vigleik}.

\subsection{Vanishing results}

In this section, we show that there are no nontrivial twistings of Morava K-theories $K(n)$ by $K(\Z, m)$ when $m>n+2$.  
Our tools 
are the computations of Ravenel-Wilson \cite{rw} and the obstruction theory described in section \ref{sec obs} above.

We recall from \cite{rav_loc, Bo} that  the Bousfield class $\langle E\rangle$ of a spectrum $E$ is the collection of $E$-acyclic spectra 
$\langle E \rangle =\{ X~:~E \wedge X=0\}$. The collection of Bousfield classes form a poset, wherein $\langle E \rangle \leq \langle F \rangle$  if 
$\langle E \rangle \supseteq \langle F \rangle$, so that if $F \wedge X =0$, then $E \wedge X$ is also $0$.

\begin{thm} \label{vanish_thm}

Let $E$ be an $A_\infty$ ring spectrum whose Bousfield class
 is $\langle E \rangle \leq \langle K(n) \rangle$.  Then if $m > n+2$, the constant function $c: K(\Z, m) \to *$ induces 
an isomorphism 
$$tw_E(*) \cong tw_E(K(\Z, m)).$$
Hence there is a single (trivial) twisting of $E$ by $K(\Z, m)$.

\end{thm}

\begin{proof}

We use the techniques developed above for $Z = \Omega K(\Z, m) = K(\Z, m-1)$.  We need to compute the Hochschild 
cohomology groups 
$$HH^{k}(E_* Z, \Omega^s E_*)  \cong \Ext^k_{E_* K(\Z, m-1)}(E_*, \Omega^s E_*)\;.$$

Now Ravenel-Wilson and Johnson-Wilson \cite{rw, jw} show that $c$ induces a ring isomorphism $c_*: K(n)_*(K(\Z, m-1)) \cong K(n)_*$.  Since 
$\langle E \rangle \leq \langle K(n) \rangle$, we must then have that $E_*(K(\Z, m-1)) \cong E_*$. Then the 
 obstruction groups vanish, since the base ring is $E_*$ itself.  
We conclude that there is a unique $A_\infty$ map
$$\Sigma^\infty K(\Z, m-1)_+ \to E$$
up to homotopy, since there is a unique $E_*$-algebra map $E_* \to E_*$.   Of course, there is an obvious such map, 
given as the composite
$$\xymatrix@1{\Sigma^\infty K(\Z, m-1)_+ \ar[r]^-c & S^0 \ar[r]^-\eta & E}$$
of the constant map and the unit.  Thus,
$$\pi_0 \Map_{A_\infty}(\Sigma^\infty K(\Z, m-1)_+, E) \cong \pi_0 \Map_{A_\infty}(S^0, E)$$
is a single element, and therefore so too is $tw_E(K(\Z, m))$.

\end{proof}

There is no need to restrict our attention in this proof to $K(\Z, m)$.  Indeed, any space with finitely many 
nonzero homotopy 
groups, all concentrated in dimensions strictly greater than $n+2$ will have trivial $K(n)_*$ (see, e.g., \cite{hrw}), 
and so the same proof applies.  

\subsection{Existence results}
\label{sec exi}

Let us now examine the boundary case of twistings of Morava K-theory $K(n)$, that is,
 twisting by $K(\Z, n+2)$.  We will require the following computations from Ravenel-Wilson, Johnson-Wilson, and Hopkins-Lurie\footnote{The original reference \cite{rw} yields these results at odd primes.  This is briefly extended to $p=2$ in the appendix to \cite{jw}. See also Theorem 2.4.10 of \cite{ambidexterity} for a more recent point of view on these computations.} \cite{rw, jw, ambidexterity}:

\begin{itemize}

\item $K(n)^* K(\Z, n+1)$ is the power series ring $K(n)_*[[x]]$, where $|x| = 2\frac{p^n-1}{p-1}$.

\item For each integer $k \geq 0$, let $R(b_k)$ be the ring 
$$R(b_k) := K(n)_*[b_k] / (b_k^p - (-1)^{n-1} v_n^{p^k} b_k) = \F_p[b_k, v_n^{{\pm 1}}] / (b_k^p - (-1)^{n-1} v_n^{p^k} b_k),$$
where the class $b_k$ has dimension $2p^k\frac{p^{n}-1}{p-1}$ and is dual to the class $(-1)^{k(n-1)} x^{p^k}$.  Then 
$$K(n)_*(K(\Z, n+1)) = \bigotimes_{k\geq 0} R(b_k).$$

\end{itemize}

\begin{rem}

In Ravenel-Wilson's notation from Theorems 12.1 and 12.4 of \cite{rw}, $x = x_S$, where $S=(1,2,\dots, n-1)$, and $b_k = b_{J}$, where $J=(nk, 1, 2, \dots, n-1)$.

\end{rem}

\begin{thm} \label{existence_thm}

The components of the space $\Map(K(\Z, n+2), BGL_1 K(n))$ are contractible.  Furthermore,

\begin{enumerate}

\item For $p>2$ there is a single trivial twist of $K(n)$ by $K(\Z, n+2)$.

\item For $p=2$, the set $tw_{K(n)}(K(\Z, n+2))$ is a group isomorphic to the $2$-adic integers, $\Z_2$.

\end{enumerate}

\end{thm}

\begin{proof}

As in Theorem \ref{vanish_thm}, we need to understand the obstruction groups
$$\Ext^k_{K(n)_*K(\Z, n+1)}(K(n)_*, K(n)_*[-s])\;.$$
According to Corollary \ref{obs_cor}, these obstruction groups are part of the $E_2$-term of the cobar spectral sequence converging to 
$K(n)^*K(\Z, n+2) \cong K(n)^*$.  That spectral sequence collapses (see the proof of Theorem 12.3 in \cite{rw}), 
and the target is $K(n)_*$, occurring in filtration $k=0$.  Consequently, the obstruction groups vanish, and
 we conclude that
$$tw_{K(n)}(K(\Z, n+2)) = \Hom_{\textit{$K_{n*}$-alg}}(K(n)_*K(\Z, n+1), K(n)_*)\;.$$
This also proves that the space of twistings is homotopically discrete.

To dispose of the cases $p>2$, we note that every degree-preserving map $K(n)_*K(\Z, n+1) \to K(n)_*$ must carry each $b_k$ to $0$, as there are no nonzero classes in the target in degree $2p^k\frac{p^{n}-1}{p-1}$.

For $p=2$, we first note that since 
\begin{equation}
K(n)_*(K(\Z, n+1)) = \bigotimes_{k\geq 0} R(b_k),
\label{eq Rbk}
\end{equation}
a $K(n)_*$-graded algebra homomorphism $f_{\unk} : K(n)_*(K(\Z, n+1)) \to K(n)_*$ is uniquely specified by a potentially infinite sequence of increasing, non-negative integers $\unk = k_1, k_2, k_3,\dots$; then $f_{\unk}$ carries each $b_{k_i}$ to $v_n^{2^{k_i}}$, and all other $b_j$ to $0$.  

Notice that the set of algebra maps in question is contained in the larger set of \emph{$K(n)_*$-module} maps 
$$\Hom_{K(n)_*-mod}(K(n)_*K(\Z, n+1), K(n)_*).$$
By the universal coefficient theorem, this is isomorphic to $K(n)^* K(\Z, n+1)$.    We are looking for maps $K(n)_*K(\Z, n+1) \to K(n)_*$ which do not shift degree.  These lie in the subring $K(n)^0 K(\Z, n+1)$, which is
$$K(n)^0 K(\Z, n+1) = \F_2[[y]], \mbox{ where $y = v_n^{-1} x$.}$$
Under this identification, $f_{\unk}$ is the product 
$$f_{\unk} = (1 + y^{2^{k_1}}) (1 + y^{2^{k_2}}) (1 + y^{2^{k_3}}) \cdots.$$
From this, it is easy to check that the map $\Hom_{\textit{$K_{n*}$-alg}}(K(n)_*K(\Z, n+1), K(n)_*) \to \Z_2$ which carries $f_{\unk}$ to $2^{k_1} + 2^{k_2} + \dots$ is an isomorphism.

\end{proof}

\begin{rem}[Monoid structure on twists]  Since $K(n)$ is an $A_\infty$ ring spectrum, 
$GL_1 K(n)$ forms an $A_\infty$-monoid and consequently we may form $BGL_1 K(n)$.  
However, $K(n)$ is known not to be an $E_m$-ring spectrum for any $m>1$.  Consequently, 
$BGL_1 K(n)$ is not an $H$-space, so it admits no further deloopings.  Consequently, it is unexpected that maps into it (i.e., twistings) should form a group. 

However, we will show in section \ref{e_section} that these twists are descended from Morava $E$-theory, which admits an $E_\infty$ ring structure, and thus its set of twists admit a group structure.  The multiplication described above is inherited from that structure.

\end{rem}

\begin{defn}

The \emph{universal twisting} of $K(n)$ by $K(\Z, n+2)$ is the topological generator $u$ of the group of twists
$$tw_{K(n)}(K(\Z, n+2)) \cong \Z_2.$$
That is, $u = 1+y$.

\end{defn}

\subsection{Odd primes and two-periodic Morava K-theory} \label{odd_primes_section}

In the previous section, the non-existence of twists of the Morava K-theories associated to odd primes was a function of the sparsity of their homotopy.  There is a variant on the spectrum $K(n)$ (which we will notate as $K_n$) whose homotopy is 2-periodic; in this section we will construct twistings of $K_n$ for all primes.  

Now, $K_n$ is an $A_\infty$ ring spectrum whose homotopy groups are given by
$$\pi_* K_n \cong \F_{p^n}[u^{\pm 1}]\;,$$
where $u$ has dimension 2.  This may be constructed as the quotient $E_n / \mathfrak{m}$, where $E_n$ is the Morava $E$-theory associated to the Honda formal group law (see section \ref{e_section} for details), and $\mathfrak{m} = (p, u_1, \dots, u_{n-1})$ is the maximal ideal in its homotopy.  

There is a map of ring spectra $K(n) \to K_n$ which presents the latter as a free module over the former; in homotopy, the map $\F_p[v_n^{\pm 1}] \to \F_{p^n}[u^{\pm 1}]$ is the natural inclusion of $\F_p$ into $\F_{p^n}$, and carries $v_n$ to $u^{p^{n}-1}$.  The flatness of this map extends the results of \cite{rw, jw} to give
$$\begin{array}{ccc}
K_n^* K(\Z, n+1) = K_{n *}[[x]] & {\rm and} & K(n)_*(K(\Z, n+1)) = \bigotimes_{k\geq 0} R^{\pm}_{K_n}(b_k)\;,
\end{array}$$
where
$$R^{\pm}_{K_n}(b_k) = \F_{p^n}[b_k, u^{\pm}] / (b_k^p - (-1)^{n-1} u^{p^k(p^n-1)} b_k)\;.$$

When $n$ is even, we note that $\F_{p^n}$ (which consists of $0$ and $p^n-1^{\rm st}$ roots of unity) contains a primitive $2p-2^{\rm nd}$ root of unity, $\xi$.  Define
$$c_k := \left\{ \begin{array}{ll} b_k, & \mbox{$n$ is odd} \\ \xi b_k, & \mbox{$n$ is even,} \end{array} \right.$$
so that in fact we always have
\beqn \label{Rck_eqn} R^{\pm}_{K_n}(b_k) = R^+_{K_n}(c_k) := \F_{p^n}[c_k, u^{\pm}] / (c_k^p - u^{p^k(p^n-1)} c_k)\;. \eeqn
With these data in hand, essentially the same argument as in Theorem \ref{existence_thm} gives:

\begin{thm} \label{existence_p_thm}

The components of the space $\Map(K(\Z, n+2), BGL_1 K_n)$ are contractible.  Furthermore, the set $tw_{K_n}(K(\Z, n+2))$ of components is a group isomorphic to $\Hom_{\textit{$K_{n*}$-alg}}({K_n}_*K(\Z, n+1), {K_n}_*)$ which is in turn isomorphic to the $p$-adic integers, $\Z_p$.

\end{thm}

To see that $\Hom_{\textit{$K_{n*}$-alg}}({K_n}_*K(\Z, n+1), {K_n}_*) \cong \Z_p$, one can use the work of Buchstaber-Lazarev or Hopkins-Lurie \cite{BL, ambidexterity} to identify the $p$-divisible group of $K(n)_*K(\Z, n+1)$ as the top exterior power of the formal group associated to $K_n$.  More concretely, and following the proof in Theorem \ref{existence_thm}, the element $a = \sum a_i p^i \in \Z_p$ corresponds to the element
$$(1+y)^a = \prod_i (1+y^{p^i})^{a_i} \in \Hom_{\textit{$K_{n*}$-alg}}({K_n}_*K(\Z, n+1), {K_n}_*)\subseteq K_n^* K(\Z, n+1).$$

This theorem can in fact be reduced under a Galois action to give corresponding results for a 2-periodic form of $K_n$ whose zeroth homotopy is any extension of $\F_p$ containing $\xi$.  

\subsection{Non-splitting of the space of twists}

One might hope that Theorem \ref{existence_thm} can be extended to show that $K(\Z, n+2)$ is a factor of 
$BGL_1 K(n)$ as, for instance, is the case with topological K-theory (cf. Sec. \ref{sec twistings}): 
$$BGL_1(K) = K(\Z, 3) \times K(\Z / 2, 1) \times B^2 SU.$$
For $n>1$, this is quickly seen to be false by the fact that 
$$\pi_i(GL_1 E) = \left\{ \begin{array}{ll}
                                           \pi_0(E)^{\times}, & i=0 \\
                                           \pi_i(E), & i>0.
                                           \end{array} \right.$$
Thus the universal cover of $BGL_1 K(n)$ is $(2^{n+1}-2)$-connected, so $BGL_1 K(n)$ cannot contain a $K(\Z, n+2)$ factor.  The same result holds for $BGL_1 K_n$ since the map $K(\Z, n+2) \to BGL_1 K_n$ cannot split in $\pi_{n+2}$, the codomain being either $0$ or $\F_{p^n}$.

\begin{prop}
Morava K-theory does not admit a determinantal twist. 
\end{prop}

There is a partial exception in the case $n=1$.  Recall that $K_1$ is a mod $p$ version of topological K-theory with $K_{1 *}=\F_{p}[u^{\pm 1}]$, where $u$ is a reduction of the usual Bott generator of degree two.  While it is impossible to split $K(\Z, 3)$ off of $BGL_1 K_1$, it is apparent that one can do so on the integral form of $K_1$, namely $K$.

\subsection{Universal and non-universal twists}

The results of Theorems \ref{existence_thm} and \ref{existence_p_thm} allow us to define, for any twist $v \in tw_{K_n}(K(\Z, n+2))$ and a class $H \in H^{n+2}(X)$, 
the twisted Morava K-theory $K_n^*(X; {v(H)})$.  In this section, we explore how the resulting group 
changes as we vary the twist $v$.

Let $d\in \Z_{> 0}$, and consider the diagram
$$\xymatrix{
\Omega X \ar[r]^-{\Omega H} & K(\Z, n+1) \ar[r]^-{\Delta_d} \ar[dr]_-{d} & K(\Z, n+1)^{\times d} 
\ar[d]_-{m_d} \ar[r]^-{v^{\times d}} & GL_1 K_n^{\times d} \ar[d]^-{m_d} \\
 & & K(\Z, n+1) \ar[r]_-{v} & GL_1 K_n\;. 
}
$$
Here $\Delta_d$ is an iterated diagonal, $m_d$ an iterated product (in a loop space), and $d: K(\Z, n+1) \to K(\Z, n+1)$ 
represents $d \in \Z = H^{n+1}(K(\Z, n+1), \Z)$.  The square commutes up to homotopy  because $v$ is an $A_\infty$ map.  
The same diagram holds for $d=0$, if we choose to interpret $\Delta_0$ as the constant map and $m_0$ as the unit in a loop space.

Applying the classifying space functor $B$ and passing along the top right of the diagram gives the 
twist $v^d(H)$ (where the power $v^d$ is performed in the group of twists).  Passing along the bottom left gives $v(dH)$.  We conclude that all twisted $K_n$ may be obtained from the universal twisting:

\begin{prop}

If $u \in tw_{K_n}(K(\Z, n+2))$ is the universal twist and $H \in H^{n+2}(X)$, there is an isomorphism
$$K_n^*(X; {u^d(H)}) \cong K_n^*(X; {u(dH)})$$
for each $d \in \Z$.

\end{prop}

With this result in mind we ignore all the non-universal twists, and make the following definition:

\begin{defn} \label{univ_defn}

For $H \in H^{n+2}(X)$, define the \emph{$H$-twisted Morava K-theory} as $K_n^* (X; H) := K_n^* (X; u(H))$, 
where $u$ is the universal twist.

\end{defn}

One may of course make the same definition at $p=2$ of $H$-twisted $K(n)$, $K(n)^*(X; H)$ using a topological generator $u$ of $\Z_2$.  

\section{Properties of twisted Morava K-theory} \label{prop_section}

\subsection{Basic properties}

Twisted Morava K-theory satisfies the axioms of Definition \ref{gen_defn}. Furthermore, we have the following basic properties, analogous to those detailed in the case of twisted K-theory in
\cite{BCMMS, MSt, BEM}:

\begin{thm}[Properties of twisted Morava K-theory] \label{basic_thm} Let $K(n)^*(X; H)$ be twisted Morava K-theory of a space $X$ with twisting class $H$. Then:

 1. ({\it Normalization}) If $H=0$ then $K(n)^*(X; H) \cong K (n)^*(X)$.

 2. ({\it Module property}) $K(n)^*(X; H)$ is a module over $K(n)^*(X)$. 

3. ({\it Cup product}) There is a cup product homomorphism 
$$
K(n)^p(X; H) \otimes K(n)^q(X; H') \longrightarrow K(n)^{p+q}(X; H + H').
$$
which makes $\oplus_{H} K(n)^*(X;H)$ into an associative ring (where $H$ ranges over all of $H^{n+2}(X; \Z)$).

4. ({\it Naturality}) If $f: Y \to X$ is a continuous map, then there is a 
homomorphism 
$$
f^*: K (n)^*(X; H) \to K (n)^*(Y; f^*H).
$$
\label{thm prop}
\end{thm}

\begin{proof}
If $H=0$, the $GL_1 K(n)$-bundle $P_H$ over $X$ determined by $H$ is trivial.  Thus the Thom spectrum $X^{ H}$ that it defines is, too: $X^{ H} \simeq X_+ \wedge K(n)$.  The homotopy of $F_{K(n)}(X^{u\circ H}, K(n))$ is thus precisely the untwisted $K(n)$-cohomology.  This gives property 1.  Property 4 is standard for twisted cohomology theories, and property 2 follows 
from properties 1 and 3.

To prove property 3, first note that if $pr_i$ denote the two projections,
$$\Delta^*(pr^*_1(H) + pr_2^*(H')) = H+H'$$
where $\Delta:X \to X \times X$ is the diagonal.  Thus $\Delta$ induces a natural map on Thom spectra 
$$\widehat{\Delta}: X^{H+H'} \to (X\times X)^{pr^*_1(H) + pr_2^*(H')} \simeq X^{H} \wedge_{K(n)} X^{H'}.$$
Applying $F_{K(n)}( - , K(n))$, we have
$$\small{\xymatrix@1{F_{K(n)}(X^{H+H'}, K(n)) & F_{K(n)}(X^{H} \wedge_{K(n)} X^{H'}, K(n)) \ar[l]_-{\widehat{\Delta}^*} & F_{K(n)}(X^H, K(n)) \wedge F_{K(n)}(X^{H'}, K(n)) \ar[l]_-{m^*}},}$$
where $m$ is given by multiplication in $K(n)$.  The induced map in homotopy ${\widehat{\Delta}^*} m^*$ induces the desired cup product; its associativity and unitality follow from the corresponding properties of $m$ and $\Delta$.

\end{proof}

\begin{rem}
There was nothing particular about $K(n)$ in this proof; Theorem \ref{thm prop} will also hold for the twisted Morava E-theory we define in section \ref{e_section}, as well as $K_n$ at any prime.  
\end{rem}

\subsection{A universal coefficient theorem} 
\label{khorami_section}

In this section we will prove an analogue of a theorem of Khorami's \cite{khorami} for twisted Morava K-theory.  Let $H \in H^{n+2}(X)$, and let $p: P_H \to X$ be the principal $K(\Z, n+1)$-bundle over $X$ corresponding to $H$.  The results of this section will be stated for the ``big" Morava K-theory $K_n$, but also hold for $K(n)$ when $p=2$.

\begin{thm} \label{khorami_thm}

There is an isomorphism
$${K_n}_*(X;H) = {K_n}_*(P_H) \otimes_{{K_n}_*(K(\Z, n+1))} {K_n}_*\;.$$

\end{thm}

We recall from section \ref{odd_primes_section} (cf. expression \eqref{Rck_eqn}) that 
${K_n}_*(K(\Z, n+1)) = \bigotimes_{k\geq 0} R(c_k)$,
where $R(c_k) := {K_n}_*[c_k] / (c_k^p - u^{p^k(p^n-1)} c_k)$, and the class $c_k$ has dimension $2p^k\frac{p^{n}-1}{p-1}$.  For brevity, write 
$$u_k := u^{p^k\frac{p^n-1}{p-1}},$$
so that $c_k^p = u_k^{p-1} c_k$, and define two cyclic $R(c_k)$-modules $M_k$ and $N_k$ by
$$\begin{array}{ccc}
M_k := R(c_k) / (c_k) & {\rm and} & N_k := R(c_k) / (c_k -u_k).
\end{array}$$

\begin{lem} \label{vanish_lem}

For $i>0$ and any $R(c_k)$-module $Q$, 
$$\Tor_i^{R(c_k)}(Q, M_k) = 0 = \Tor_i^{R(c_k)}(Q, N_k).$$

\end{lem}

\begin{proof}

An explicit periodic free resolution of $M_k$ is given by
$$\xymatrix@1{
0 & M_k \ar[l] & R(c_k) \ar[l]_{\epsilon} & R(c_k) \ar[l]_{~c_k} && R(c_k) \ar[ll]_-{c_k^{p-1} - u_k^{p-1}} & R(c_k) 
\ar[l]_{~c_k} && R(c_k) \ar[ll]_-{c_k^{p-1} - u_k^{p-1}} & \cdots, \ar[l]
}$$
so 
$$\Tor_{odd}^{R(c_k)}(Q, M_k) = \{q \in Q \; | \; c_k q = 0 \} \, /  \,(c_k^{p-1}-u_k^{p-1})Q,$$
and in positive degrees
$$\Tor_{even}^{R(c_k)}(Q, M_k) = \{q \in Q \; | \; (c_k^{p-1}-u_k^{p-1}) q = 0 \} \, / \, c_k Q.$$
However, in $R(c_k)$, the formula
$$u_k^{p-1} = c_k^{p-1} - (c_k^{p-1}-u_k^{p-1})$$
holds, so if $c_k q =0$, then $q = -u_k^{1-p}(c_k^{p-1}-u_k^{p-1})q \in (c_k^{p-1}-u_k^{p-1})Q$.  Thus the odd Tor groups 
for $M_k$ are zero.  A similar computation gives the result for the even Tor groups.  

Define
$$n_k = \frac{c_k^p-u_k^{p-1}c_k}{c_k-u_k} = c_k(c_k^{p-2} + c_k^{p-3}u_k + \dots + c_k u_k^{p-3} + u_k^{p-2}).$$
Then a resolution of $N_k$ is given by 
$$\xymatrix@1{
0 & N_k \ar[l] & R(c_k) \ar[l]_{\epsilon} & R(c_k) \ar[l]_{~c_k-u_k} & R(c_k) \ar[l]_-{n_k} & R(c_k) 
\ar[l]_{~c_k-u_k} & R(c_k) \ar[l]_-{n_k} & \cdots, \ar[l]
}$$
so that 
$$\Tor_{odd}^{R(c_k)}(Q, N_k) = \{q \in Q \; | \; (c_k-u_k) q = 0 \} \, /  \,n_k Q.$$
But if $c_k q=u_k q$, then $n_k q = (p-1) u_k^{p-1} q$, so $q \in n_k Q$, and this group is $0$.  Furthermore, in positive degrees,
$$\Tor_{even}^{R(c_k)}(Q, N_k) = \{q \in Q \; | \; n_k q = 0 \} \, / \, (c_k-u_k) Q.$$
Take $q$ with $n_k q = 0$.  Then 
$$0 = u_k^{1-p} n_k q = \sum_{r=1}^{p-1} u_k^{-r} c_k^r q,$$
so
$$q = \frac{-(c_k-u_k)}{u_k(1-\frac{c_k}{u_k})} q = -(c_k-u_k) u_k^{-1} (1+\sum_{r> 0} u_k^{-r} c_k^r)q =  -(c_k-u_k) u_k^{-1}q \in (c_k-u_k) Q.$$

\end{proof}

\noindent{\it Proof of Theorem \ref{khorami_thm}.}
By construction, the twisted ${K_n}$-homology of $X$ is given by the homotopy groups of the quotient
\begin{eqnarray*}
{K_n}_*(X, H) & := & \pi_*((P_H)_+ \wedge_{K(\Z, n+1)} {K_n}) \\
 & = & \pi_*(({K_n} \wedge (P_H)_+) \wedge_{({K_n} \wedge K(\Z, n+1))} {K_n}).
\end{eqnarray*}
From this description, we immediately get a bar spectral sequence
$$\Tor_*^{{K_n}_*(K(\Z, n+1))}({K_n}_* P_H, {K_n}_*) \implies {K_n}_*(X, H).$$
By Ravenel-Wilson's results, however, the $E_2$-term may be decomposed as
$$\Tor_*^{{K_n}_*(K(\Z, n+1))}({K_n}_* P_H, {K_n}_*) = \bigotimes_{k =0}^\infty \Tor_*^{R(c_k)}({K_n}_* P_H, {K_n}_*).$$

Finally, it is apparent from the description of the universal twist, as given by the element 
$u=1+y \in {K_n}^*K(\Z, n+1)$, 
that ${K_n}_*$ is made into a ${K_n}_*K(\Z, n+1)$-module by letting $c_k$ act as $0$ for $k>0$ and letting $c_0$ act as $u_0 = u^{\frac{p^n-1}{p-1}}$.  
Thus this $E_2$-term is
$$\Tor_*^{R(c_0)}({K_n}_* P_H, N_0) \otimes \bigotimes_{k =1}^\infty \Tor_*^{R(c_k)}({K_n}_* P_H, M_k).$$
By Lemma \ref{vanish_lem}, all of the higher Tor terms vanish.

\qed

Note that a consequence of this proof is also an explicit description of the twisted Morava K-theory.  Namely, 
${K_n}_*(P_H)$ comes equipped with operations $c_k$ via the action of $K(\Z, n+1)$ on $P_H$. 

\begin{cor}

There is an isomorphism
$${K_n}_*(X, H) \cong {K_n}_*(P_H) / (c_0-u^{\frac{p^n-1}{p-1}}, c_1, c_2, \dots).$$

\label{khorami_cor}
\end{cor}

\section{A twisted Atiyah-Hirzebruch spectral sequence} \label{ahss_section}

In this section, we work entirely at the prime 2, allowing us to concentrate on twisted $K(n)$, rather than $K_n$.  It is quite likely that an odd-primary analogue of the main result, Theorem \ref{ahss_thm}, also holds.  However, there do appear to be some subtle issues at odd primes; for instance, an exact analogue of the differential in Theorem \ref{ahss_thm} cannot possibly be true, owing to the fact that $Q_n$ and $Q_{n-1} \dots Q_1(H)$ are of different dimensions when $p$ is odd.  

The cellular filtration of a CW complex gives rise to the Atiyah-Hirzebruch spectral sequence (AH) for Morava K-theory.  The same holds in the twisted case:

\begin{thm} \label{ahss_thm}

For $H \in H^{n+2}(X)$, there is a spectral sequence converging to $K(n)^*(X; H)$ with 
$E_2^{p, q} = H^p(X, K(n)^q)$.  
The first possible nontrivial differential is $d_{2^{n+1}-1}$; this is given by
$$d_{2^{n+1}-1}(xv_n^k) = (Q_n(x) + (-1)^{|x|}  x\cup (Q_{n-1} \cdots Q_1(H)))v_n^{k-1}.$$

\end{thm}

Here $Q_n$ is the $n^{\rm th}$ Milnor primitive at the prime $2$.  It may be defined inductively as
$Q_0 = Sq^1$, the Bockstein operation, and 
$Q_{j+1}=Sq^{2^j}Q_j - Q_j Sq^{2^j}$, where 
$Sq^j: H^n(X;\F_2) \to H^{n+j}(X;\F_2)$ 
is the $j$-th Steenrod square.
These operations are derivations\footnote{The signs are of course irrelevant at $p=2$, but will become appropriate in the integral version in section \ref{sec int}.}
 \cite{Mi}
$$
Q_j(xy)=Q_j(x)y + (-1)^{|x|}x Q_j(y)\;.
$$

The proof of the existence of this spectral sequence is identical to the approach taken in \cite{as1} \cite{as2}.  The bulk 
of our work in the next sections, will be devoted to computing the first differential in this spectral sequence.

Some of the properties of the differentials in the spectral sequence for twisted K-theory are discussed 
in \cite{as1} \cite{as2} \cite{KS2}.  Many of the analogous properties also hold in this setting.  In particular, we have:

\begin{prop} For the twisted Atiyah-Hirzebruch spectral sequence for twisted Morava K-theory
\begin{enumerate}
\item ({\it Linearity}) Each differential $d_i$ is a $K(n)_*$-module map.
\item ({\it Normalization}) The twisted differential with a zero twist reduces to the untwisted differential (which may be zero).
\item ({\it Naturality}) If $f:Y \to X$ is continuous, $f^*$ induces a map of spectral sequences from the tAH for $(X, H)$ to the tAH for $(Y, f^* H)$.  On $E_2$-terms it is induced by $f^*$ in cohomology, and on 
$E_\infty$-terms it is the associated graded map induced by $f^*$ in twisted Morava K-theory.
\item ({\it Module}) The tAH for $(X, H)$ is a spectral sequence of modules for the untwisted AHSS for $X$.  Specifically, $d_i(ab) = d_i^u(a) b + (-1)^{|a|} a d_i(b)$ where $a$ comes from the untwisted spectral sequence (with differentials $d_i^u$).
\end{enumerate}
\end{prop}

\subsection{Preliminary spectral sequences computations}

We will need a cohomology representative of the class $x$ that figured so prominently in our construction of the twisting
(cf. sections \ref{sec exi} and \ref{khorami_section}).

\begin{lem} \label{rep_lem}

In the $K(n)$-Atiyah-Hirzebruch spectral sequence for $K(\Z, n+1)$,
$$H^*(K(\Z, n+1), \F_2) \otimes K(n)_* \implies K(n)^*(K(\Z, n+1)) = K(n)_*[[x]],$$
the class $x$ is represented by the class $Q_{n-1} \dots Q_1(j)$, where $j \in H^{n+1}(K(\Z, n+1), \F_2)$ is the fundamental class.

\end{lem}

\begin{proof}

This follows from Ravenel-Wilson's construction of $b_0$, the class dual to $x$.  The former is defined as
$$b_0 = \delta_*(a_{(0)} \circ a_{(1)}\circ \cdots \circ a_{(n-1)})$$
where $\circ$ is the product in the Hopf ring structure on $K(n)_*(K(\Z / 2, *))$, and $\delta: K(\Z / 2, n) \to K(\Z, n+1)$ 
is the map realizing the Bockstein in cohomology.  The natural map $H_*(K(\Z / 2, n), \F_2) \to \mathcal{A}_*$ (the dual 
Steenrod algebra) carries the AH representative of 
$a_{(0)} \circ a_{(1)}\circ \cdots \circ a_{(n-1)}$ to $\tau_0 \cdot \tau_1 \cdots \tau_{n-1}$ (see \cite{rwy}, section 8.3.1).  

Therefore, $b_0$ is represented by a class in $H_*(K(\Z, n+1), \F_2)$ mapping to $\tau_1 \cdots \tau_{n-1}$; dually, 
$x$ is represented by  $Q_{n-1} \dots Q_1(j)$.

\end{proof}

We will also need to study the Atiyah-Hirzebruch-Serre spectral sequence (AHS) for the principal $K(\Z, n+1)$-fibration 
$p: P_H \to X$ associated to the twisting class $H$:
$$H^*(X, K(n)^* K(\Z, n+1)) \implies K(n)^*(P_H).$$

\begin{lem} \label{diff_lem}

The (transgression) differential in AHS is
$$d^{AHS}_{2^{n+1}-1}(x) = Q_{n-1} \dots Q_{1} (H) \mod (x).$$

\end{lem}

\begin{proof}

It suffices (via pullback along $H$) to show this in the universal case where 
$X = K(\Z, n+2)$, and $H = \iota\in H^{n+2}(K(\Z, n+2), \F_2)$ is the fundamental class.  
In this case $P_H \simeq *$, so the AHS is of the form
$$H^*(K(\Z, n+2), \F_2) \otimes K(n)_*[[x]] \implies K(n)^*.$$

We compare the Atiyah-Hirzebruch-Serre spectral sequences for three different cohomology theories: $K(n)$, its 
connective analogue $k(n)$, and mod $2$ cohomology $H^*$.  There are natural transformations 
$$\xymatrix@1{K(n) & k(n) \ar[l]_-{\ell} \ar[r]^-{P} & H^*.}$$ 
The first transformation corresponds to localization at $(v_n)$, and the second is 
the zeroth Postnikov section of $k(n)$.  These give rise to maps of AHS:
$$\xymatrix{
M \otimes K(n)^*K(\Z, n+1) \ar@{=>}[d] & M \otimes k(n)^*K(\Z, n+1)  \ar[l]_-{\id \otimes \ell} \ar[r]^-{\id \otimes P} 
\ar@{=>}[d]  & M \otimes H^*K(\Z, n+1) \ar@{=>}[d] \\
K(n)^* & k(n)^* \ar[l]^-{\ell} \ar[r]_-{P} & H^* 
}$$
where $M = H^*(K(\Z, n+2), \F_2)$.


The class $x \in K(n)^{2^{n+1}-2} K(\Z, n+1)$ lifts to a
class $X \in k(n)^{2^{n+1}-2}K(\Z, n+1)$ with $\ell(X) = x$.  In fact, it actually lifts to $BP$; this is Tamanoi's ``$BP$ fundamental class" $\vartheta$ of \cite{tamanoi97}; see also Section \ref{conj_section} below.   Furthermore, 
Lemma \ref{rep_lem} implies that $P(X) = Q_{n-1} \dots Q_1(j)$.  Since the transgression commutes with
 Steenrod operations, we have
\begin{eqnarray*}
P(d^{AHS}_{2^{n+1}-1}(X)) & = & d^{AHS}_{2^{n+1}-1}(PX) \\
 & = & d^{AHS}_{2^{n+1}-1}(Q_{n-1} \dots Q_1(j)) \\
 & = & Q_{n-1} \dots Q_1(d^{AHS}_{{n+2}}(j)) \\
 & = & Q_{n-1} \dots Q_1(\iota).
\end{eqnarray*}
In bidegree $(2^{n+1}-1, 0)$, $P$ is an isomorphism, so in the $k(n)$-AHS we may 
conclude that $d^{AHS}_{2^{n+1}-1}(X) = Q_{n-1} \dots Q_1(\iota)$.  Furthermore, in that bidegree, 
$\ell$ is an isomorphism mod $(x)$ and so we see that in the $K(n)$-AHS
$$d^{AHS}_{2^{n+1}-1}(X) = Q_{n-1} \dots Q_1(\iota) \mod (x).$$

\end{proof}

\subsection{The first differential}

Because of the sparsity of the homotopy in $K(n)_*$, the first non-vanishing differential in tAH,
the twisted AH spectral sequence,  is $d_{2^{n+1}-1}$.

\begin{lem} \label{tah_lem}

There is a natural transformation $\phi_n : H^{n+2}(-, \Z) \to H^{2^{n+1}-1}(-, \F_2)$ with the 
property that, in the twisted Atiyah-Hirzebruch spectral sequence converging to $K(n)^*_H(X)$,
$$d_{2^{n+1}-1}(xv_n^k) = (Q_n(x) + (-1)^{|x|}  x \cup \phi_{n}(H))v_n^{k-1} $$
for $x \in H^*(X, \F_2)$ and $k \in \Z$.

\end{lem}

\begin{proof}

Since this is a spectral sequence of modules over $K(n)_*$, it is sufficient to show that this is true for $k=0$. 
 Further, since the twisted AHSS is a module over the untwisted AHSS, we can compute 
$$d_{2^{n+1}-1}(x) = d_{2^{n+1}-1}(x\cdot 1) = d^u_{2^{n+1}-1}(x) \cdot 1 + (-1)^{|x|} x \cdot d_{2^{n+1}-1}(1),$$
where $d_{2^{n+1}-1}^u$ denotes the differential in the untwisted AHSS.  But this is known \cite{Y} to be given by 
$Q_n$.  So it is enough to show that $d_{2^{n+1}-1}(1) = \phi_{n}(H) v_n^{-1}$.  Bidegree considerations imply 
that $d_{2^{n+1}-1}(1)$ must take the form $\beta v_n^{-1}$; naturality of the spectral sequence in the twist ensures 
that $\beta$ is natural in $H$, and so must be of the form $\phi_{n}(H)$.

\end{proof}

\noindent {\it Proof of Theorem \ref{ahss_thm}.}
Return to the (homological) Atiyah-Hirzebruch-Serre spectral sequence (AHS):
$$H_*(X, K(n)_* K(\Z, n+1)) \implies K(n)_*(P_H).$$
The quotient map $K(n)_*(P_H) \to K(n)_*(X, H)$ (as in section \ref{khorami_section}) induces a map of spectral sequences
$$H_*(X, K(n)_* K(\Z, n+1)) \to H_*(X, K(n)_*)$$
from the above AHS spectral sequence to the twisted Atiyah-Hirzebruch spectral sequence (tAH).  On the level of $E_2$-terms, 
it is induced by map of coefficients $K(n)_* K(\Z, n+1) \to K(n)_*$ which is give by the universal twisting $y$.  That is, 
it is the map $\otimes_k R(b_k) \to K(n)$ which sends $b_0$ to $v_n$ and all other $b_i$ to 0.  This map is surjective, 
so we can recover the differential in tAH from knowledge of the differential in AHS.  

By degree considerations, for a class $z \in H_*(X, \F_p) = E^2_{*, 0}$, we have
$$d^{AHS}_{2^{n+1}-1}(z) = a(z) v_n + b(z) b_0 + \dots,$$
where the higher order terms all involve $b_i$ for $i>0$, and $a$ and $b$ are operations on $z$ which lower degree 
by $2^{n+1}-1$.  Since AHS is a comodule over the $K(n)$-Atiyah-Hirzebruch spectral sequence for $X$, it is easy to 
see that $a = Q_n^{\vee}$ is dual to the Milnor primitive operation $Q_n$.  Mapping to $tAH$, we see
$$d^{tAH}_{2^{n+1}-1}(z) = Q_n^{\vee}(z) v_n + b(z) v_n.$$
Comparing this to the dual spectral sequence in Lemma \ref{tah_lem}, we see that $b(z) = (-1)^{|z|} z \cap (\phi_{n}(H))$.

Rewriting this result in the cohomology AHS spectral sequence
$$H^*(X, K(n)^* K(\Z, n+1)) = H^*(X) \otimes K(n)_*[[x]] \implies K(n)^*(P_H),$$
we deduce that $\phi_{n}(H) = d^{AHS}_{2^{n+1}-1}(x) \mod (x)$, for when we pair it against a homology class,
\begin{eqnarray*}
\langle d^{AHS}_{2^{n+1}-1}(x), z \rangle & = & \langle x, d^{AHS}_{2^{n+1}-1}(z) \rangle \\
 & = & \langle x, a(z) v_n + b(z) b_0 + \dots \rangle \\
 & = & b(z) \\
 & = & (-1)^{|z|} z \cap (\phi_{n}(H)).
\end{eqnarray*}
So, by Lemma \ref{diff_lem}, $\phi_{n}(H) = Q_{n-1} \cdots Q_1(H)$.

\qed

\section{Twisted Morava E-theory} \label{e_section}

In this section, we will lift the previous constructions to the (Landweber exact) Lubin-Tate cohomology theory most commonly known as Morava E-theory.  We begin with a reminder of this structure; \cite{gh, rezk_notes} are much better introductions to this material, and we will give a very terse summary of the relevant part of the latter.  

\subsection{Defining Morava E-theory}

There are several cohomology theories that compete for the title of ``Morava E-theory":

\begin{itemize}

\item $\BP{n}$, the truncated Brown-Peterson spectrum, with coefficients
$$\BP{n}_* = \Z_{(p)}[v_1, \dots, v_n].$$
This is constructed by killing the ideal $(v_{n+1}, v_{n+2}, \dots)$ in the homotopy of $BP$, the Brown-Peterson spectrum, the home of the universal $p$-typical formal group law.

\item $E(n)$, the  Johnson-Wilson spectrum, with coefficients
$$E(n)_* = \Z_{(p)}[v_1, \dots, v_{n-1}, v_n^{{\pm 1}}].$$
$E(n)$ is constructed as the localization of $\BP{n}$ at $v_n$.

\item $\widehat{E(n)}$, the completed Johnson-Wilson spectrum, with coefficients the completion of the previous at the ideal $I = (p, v_1, ..., v_{n-1})$:
$$\widehat{E(n)}_* = \Z_{(p)}[v_1, \dots, v_{n-1}, v_n^{{\pm 1}}]^{\wedge}_I.$$

\item $E(k, \Gamma)$, the Lubin-Tate spectrum, associated to the universal deformation of a formal group law $\Gamma$ over $k$.

\end{itemize}

Our techniques in this section will produce a twisting of certain Lubin-Tate spectra.  The argument also works for $\widehat{E(n)}$, though our techniques fail for $\BP{n}$ and $E(n)$, as their coefficients are not complete local rings.  We offer a conjecture in section \ref{conj_section} which provides a slightly more geometric interpretation of these twistings, and, if true, would yield twistings of $\BP{n}$ and $E(n)$.

Let us review the Lubin-Tate spectra in more detail: $k$ is a perfect field of positive characteristic $p$, and $\Gamma \in k[[x, y]]$ is a formal group law defined over $k$ of height $n$.  A deformation $(B, G, i)$ of $(k, \Gamma)$ is a complete local ring $B$ with maximal ideal $\m$, a formal group law $G$ on $B$, and a ring homomorphism $i:k \to B / \m$ with $i^*\Gamma = \pi^* G$, where $\pi: B \to B / \m$ is the quotient map.  

It is a theorem of Lubin-Tate \cite{lt} that there exists a \emph{universal} such deformation.  This is a complete local ring $A(k, \Gamma)$ with maximal ideal $\m$ such that $A(k, \Gamma) / \m$ is isomorphic to $k$.  Concretely, $A(k, \Gamma)$ is the null-graded ring
$$\begin{array}{ccc}
A(k, \Gamma) \cong \W(k)[[u_1, \dots, u_{n-1}]] & {\rm and} & \m = (p, u_1, \dots, u_{n-1}),
\end{array}$$
where $\W(k)$ is the ring of Witt vectors in $k$ (e.g., $\Z_p$ if $k=\F_p$).

This ring is equipped with a formal group law $F$ from which all deformations of $\Gamma$ are pulled back (up to $\star$-isomorphism).  Specifically, for any deformation $(B, G, i)$ of $(k, \Gamma)$, there exists a ring homomorphism $\phi: A(k, \Gamma) \to B$ (over $k$) and a unique $\star$-isomorphism $f: \phi^*F \to G$ (that is, $f = f(x)$ is an isomorphism of formal group laws over $B$ which reduces to $x$ modulo the maximal ideal of $B$).

The theorem of Goerss-Hopkins-Miller \cite{gh, gh2, rezk_notes} shows that there exists an essentially unique even periodic, $E_\infty$ ring spectrum $E(k, \Gamma)$ realizing this universal deformation; that is,
$$\begin{array}{ccc}
\pi_*E(k, \Gamma) = A(k, \Gamma)[u^{\pm1}], & {\rm and} & \Spf E(k, \Gamma)^*\C P^\infty \cong F.
\end{array}$$
$E(k, \Gamma)$ is closely related to $\widehat{E(n)}$.  In particular, when $\Gamma$ is the Honda formal group law, the map $BP_* \to E(k, \Gamma)_*$ which classifies $\Gamma$ carries $v_k$ to $u^{p^k-1}u_k$ for $k<n$, $v_n$ to $u^{p^n-1}$, and $v_m$ to $0$, when $m>n$.

Since the ideal $\m$ is generated by a regular sequence, the residue field $k[u^{\pm 1}] =  A(k, \Gamma)[u^{\pm1}] / \m$ is realized as the homotopy groups of a spectrum which we will denote by $K(k, \Gamma)$.  When $k$ is the prime field $\F_p$ and $\Gamma$ is the Honda formal group law with $p$-series $[p](x) = x^{p^n}$, $K(k, \Gamma)$ is a 2-periodic form of the ``standard" Morava $K$-theory:
$$K(k, \Gamma) \simeq K(n)[u]/(u^{p^n-1} - v_n).$$
More generally, if $k$ is a finite extension of $\F_p$, then $K(k, \Gamma)$ is a $K(n)$-algebra spectrum which is a finite rank free $K(n)$-module (see, e.g., Corollary 10 of \cite{jacob_notes}).

\begin{exmp} 

If $k=\F_p$, and $\Gamma = \mathbb{G}_m$ is the multiplicative group, then $A(k, \Gamma) = \Z_p$; the universal deformation is once again the multiplicative group (only over $\Z_p$ instead of $\F_p$).  The resulting ring spectrum $E(k, \Gamma)$ is $p$-completed K-theory, and $K(k, \Gamma) = K_1$ is mod $p$ K-theory.

\end{exmp}

One consequence of the Goerss-Hopkins-Miller theorem is that for any of these Lubin-Tate theories $E(k, \Gamma)$, the space of units, $GL_1 E(k, \Gamma)$, is not just an $A_\infty$ monoid (as above for $K(n)$), but an infinite loop space.  Its deloopings give rise (as in, e.g. \cite{abghr}) to a connective spectrum $gl_1 E(k, \Gamma)$, with $\Omega^\infty gl_1 E(k, \Gamma) = GL_1 E(k, \Gamma)$.

\begin{assump} \label{fgl_assump} We will make two main assumptions on $(k, \Gamma)$ for the rest of this paper:

\begin{enumerate} 

\item The formal group law $\Gamma$ over $k$ is $p$-typical, with $p$-series $[p](x) = x^{p^n}$.  Equivalently, $\Gamma$ is induced from the map $BP_* \to k$ carrying $v_n$ to $1$ and all other $v_i$ to $0$. Essentially, we are assuming that $\Gamma$ is isomorphic to an extension of the Honda formal group law.

\item When $n$ is even, $k$ contains a primitive $2p-2^{\rm nd}$ root of unity, $\xi$.

\end{enumerate}

\end{assump}

The first assumption is forced by our use of the results of \cite{rw, jw}, whose computations are all based on the standard Morava K-theory $K(n)$ associated to the Honda formal group law over $\F_p$.  We expect that Hopkins-Lurie's approach \cite{ambidexterity} to these computations will allow one to extend our results to a more general setting.  The second assumption is required, as in section \ref{odd_primes_section}, to deal with a troublesome sign in the Ravenel-Wilson calculations.

\begin{notation}

For brevity, we will employ the notation $E_n$ for $E(k, \Gamma)$, where $(k, \Gamma)$ is as in Assumption \ref{fgl_assump}.  We will also write $K_n$ for $K(k, \Gamma)$. 

\end{notation}

\subsection{$E_\infty$ twistings}

We now consider twists for Morava E-theory.  Our main result (Theorem 2 of the introduction) is the following:

\begin{thm} \label{e(n)_thm}

For each $n\geq 1$, there are canonical isomorphisms
$$\pi_0 \Map_{E_\infty} (K(\Z, n+2), BGL_1 E_n) \cong [\Sigma^{n+1} H\Z, gl_1 E_n] \cong \Hom_{\textit{$E_{n*}$-alg}}({E_n}_*K(\Z, n+1), {E_n}_*),$$
and the latter group is isomorphic to the $p$-adic integers $\Z_p$.

\end{thm}

Let $\varphi_n: H\Z \to gl_1 E_n$ be a topological generator; the induced map $K(\Z, n+1) \to GL_1 E_n$ allows us to twist $E_n$ by a class in $H^{n+2}$, as we did for $K(n)$.  By definition, there is an $A_\infty$ map $\pi: E_n \to {K_n}$ which quotients by the ideal $\m$.

\begin{prop} \label{nontriv_prop}

The composite map $\Omega^{\infty}(\pi \circ \varphi_n): K(\Z, n+1) \to GL_1 {K_n}$ is the universal twisting $u$.

\end{prop}

The homotopical nontriviality of $\varphi_n$ follows from this Proposition.  It also says that twisted Morava K-theory is in some sense a reduction of twisted Morava E-theory.  The proofs of these results occur after Proposition \ref{cotangent_prop}, below.

\subsection{A conjecture} \label{conj_section}

Consider the Johnson-Wilson spectrum $\BP{n}$. Note that $\BP{0} = H\Z_{(p)}$ is the $p$-local Eilenberg-MacLane spectrum.  Multiplication by $v_n$ gives a cofiber sequence
\beqn \label{sequence_eqn}
\xymatrix@1{\Sigma^{2p^n-2} \BP{n} \ar[r]^-{v_n} & \BP{n} \ar[r]^-{p_n} & \BP{n-1} \ar[r]^-{\Delta_n} & \Sigma^{2p^n-1} \BP{n} \ar[r] & \cdots}.
\eeqn
Here $p_n$ quotients by $v_n$, and $\Delta_n$ is the connecting map in the long exact sequence, and is related to the Bockstein operator $Q_n$ defined by Baker-W\"urgler in \cite{bw}.

Let us now define $\delta_{k}$ as the composite
$$\delta_{k}=\Delta_{k} \circ \cdots \circ \Delta_2 \circ \Delta_1:  \BP{0} \to \Sigma^{2\frac{p^{k+1}-1}{p-1}-(k+2)} \BP{k}.$$
The map $\delta_{k}$ is closely related to a map $\vartheta_{k+2}$ studied by Tamanoi in \cite{tamanoi97}, which he calls the \emph{BP fundamental class of $K(\Z_{(p)}, k+2)$}.

\begin{conj}[$p=2$] \label{conjecture}
{\it
Assuming that $\BP{n}$ admits the structure of an $E_\infty$ ring spectrum\footnote{This is known for $n=1$, classically, and for $n=2$, by \cite{lawson_naumann}.  The assumption is necessary in order to form $gl_1 \BP{n}$.  One could, however, relax this assumption and simply ask for an $A_\infty$ map $\underline{\BP{n-1}}_{2^{n+1}-2} \to GL_1 \BP{n}$.}, there are maps of spectra
$$\psi_n: \Sigma^{2^{n+1}-2} \BP{n-1} \to gl_1 \BP{n}$$
whose induced map in homotopy is $x \mapsto xv_n$.  That is, in all degrees $k\geq2^{n+1}-2$, $(\psi_n)_*$ is the composite
$$\xymatrix@1{ \Z_{(2)}[v_1, \dots, v_{n-1}]_{(k-2^{n+1}+2)} \ar[r]^-{\subseteq} &  \Z_{(2)}[v_1, \dots, v_{n}]_{(k-2^{n+1}+2)} \ar[r]^-{v_n} & \Z_{(2)}[v_1, \dots, v_n]_{k}}.$$
}
\end{conj}

There is some evidence for this conjecture.  For $n=1$, $\psi_1$ exists, and its infinite delooping is the 2-localization of a map $\C P^\infty \to GL_1 K$.  This is precisely the inclusion of line bundles into invertible virtual bundles which gives rise to the usual determinantal twisting of K-theory.  For $n=2$, the conjectured map is of the form $\Omega^\infty \psi_2: BU\langle 6 \rangle \to GL_1 \BP{2}$.  Such an $E_\infty$ map would be adjoint to an $E_\infty$ map of the form
$$\Sigma^\infty BU \langle 6 \rangle_+ \to \BP{2}.$$
Such a map might in principle be constructed from the $\sigma$-orientation $MU\langle 6 \rangle \to tmf$ of Ando-Hopkins-Strickland \cite{ahs} and an $\BP{2}$-Thom isomorphism.

Furthermore, in \cite{wilson}, Wilson shows that the fiber sequence associated to 
$(\ref{sequence_eqn})$ splits\footnote{This is a $2$-local version of the splitting 
$BU_\otimes \simeq \C P^{\infty} \times BSU_{\otimes}$ when $n=1$ and $k=2$.}:
$$\underline{\BP{n}}_k \simeq \underline{\BP{n-1}}_k \times \underline{\BP{n}}_{k+2^{n+1}-2}, \mbox{ for $k \leq 2^{n+1}-2$.}$$
One may derive from this (for $k=2^{n+1}-2$) a map of \emph{spaces}
$$\xymatrix@1{\underline{\BP{n-1}}_{2^{n+1}-2} \ar[r]^i & \underline{\BP{n}}_{2^{n+1}-2} \ar[r]^-{v_n} & \underline{\BP{n}}_{0}}\;,$$
where $i$ is Wilson's splitting map.  This induces the desired map in homotopy, and it is apparent that the image of the map lies in the $v_n$ component, and hence in $GL_1 \BP{n}$.  However, it is far from obvious that this is an infinite loop map.

In any case, if the conjectured map $\psi_n$ exists, we may construct the twisting 
\footnote{Or rather, its analogue for $\BP{n}$.}
$\varphi_n$ as the composite
$$\label{diagram_eqn}
\xymatrix{
\Sigma^{n+1} H\Z \ar[r] \ar[d]_-{\varphi_n} & \Sigma^{n+1} H\Z_{(2)} = \Sigma^{n+1} \BP{0} \ar[d]^-{\delta_{n-1}} &  \\
gl_1 \BP{n} &  \Sigma^{2^{n+1}-2} \BP{n-1}\;.  \ar@{-->}[l]^-{\psi_n} \\
}$$
Rather than proving the conjecture, we will show that the $E_n$-analogue of the desired map $\varphi_n$ exists, using the Andr\'{e}-Quillen obstruction theory.

\subsection{The formal group $E_n^* K(\Z, n+1)$}

To begin the obstruction theory calculations, we make use of the Bockstein spectral sequence for Morava E-theory, developed for the completed Johnson-Wilson spectra by Baker-W\"urgler in \cite{bw}.  Write $\m$ for the maximal ideal of ${E_n}_*$.  The $E_1$-page of the spectral sequence is given by
$$E_1^{s, *} = K_n^{*}(X) \otimes_{{K_n}_*} \m^{s}/\m^{s+1} \implies E_n^*(X).$$

We first note that $K_n^{*}K(\Z, n+1) \cong {K_n}_*[[x]]$; this follows from the corresponding statement for $K(n)$ by the fact that ${K_n}$ is a finite rank free $K(n)$-module.  Then the Bockstein spectral sequence for $K(\Z, n+1)$ collapses at $E_1$ since the total degree of each class is concentrated in even dimensions.  We conclude:

\begin{prop} \label{bock_prop}

There is a ring isomorphism
$$E_n^*K(\Z, n+1) \cong {E_n}_*[[\tx]],$$
where $|\tx| = 2\frac{p^n-1}{p-1}$, and projects to $x \in K_n^* K(\Z, n+1)$ modulo $\m$.

\end{prop}

We note that this ring is flat over ${E_n}_*$ and so we conclude, similarly, that
$$E_n^*(K(\Z, n+1) \times K(\Z, n+1))\cong {E_n}_*[[\ty, \tz]].$$
The multiplication map $m: K(\Z, n+1) \times K(\Z, n+1) \to K(\Z, n+1)$ therefore makes ${E_n}_*[[\tx]]$ into a formal group over ${E_n}_*$ with formal group law $F_n(\ty, \tz) := m^*(\tx)$.  This lifts a corresponding formal group law on ${K_n}_*[[x]]$.

\begin{lem} \label{def_lem}

The formal group $\Spf(K_{n *}[[x]])$ is isomorphic to the multiplicative group $\mathbb{G}_m$ over $K_{n *}$.  

\end{lem}

\begin{proof}

Ravenel-Wilson show that the Verschiebung on $K(n)^* K(\Z, n+1)$ satisfies $V(x) =  (-1)^{n+1} x$, so the same holds over $K_n^* K(\Z, n+1)$.  Again using the primitive $2p-2^{\rm nd}$ root of unity $\xi \in k$ when $n$ is even, we may define a new coordinate $z$ as $z=x$ when $n$ is odd, and $z = \xi x$ when $n$ is even.

If $n$ is odd, evidently $V(z) = z$.  If $n$ is even, the same holds: $V(z) =\xi^p V(x) = - \xi^p x = z$.  Hence the formal group law $G$ on $K_n^* K(\Z, n+1)$ satisfies $[p](z) = FV(z) = z^p$, and so is of height 1.  Now, this $G$ is classified by a map $\theta: V \to K(n)_*$, where $V=\Z_{(p)}[v_1, v_2, \dots] \cong BP_*$ supports the universal $p$-typical formal group law.  Then the equation
$$z^p = [p](z) = \sum_i^G \theta(v_i) z^{p^i}$$
(see, e.g., (A2.2.4) in \cite{ravenel}) can be solved to see that $\theta(v_i) = 0$, for $i>1$.  This immediately implies that $G = \mathbb{G}_m$.

\end{proof}

See also the lovely \cite{peterson} for another approach to this sort of result.  Now, a power series $f \in {E_n}_*[[\tx]]$ will give rise to an ${E_n}_*$-algebra map ${E_n}_* K(\Z, n+1) \to {E_n}_*$ precisely when $f$ is \emph{grouplike}; that is, when
$$m^*(f) = f \otimes f.$$
This is because $f(ab) = \sum f_i'(a) f_i''(b)$, where $m^*(f) = \sum f_i' \otimes f_i''$. 

\begin{lem}

The element $\alpha = 1+\tx \in E_n^*K(\Z, n+1)$ is grouplike.  That is, 
$$m^*(\alpha) = \alpha \otimes \alpha = 1 + \ty + \tz + \ty~\tz.$$

\end{lem}

\begin{proof}

The reduction $A(k, \Gamma) \to k$, along with Lemma \ref{def_lem} exhibit $\Spf({E_n}_*[[\tx]])$ as a deformation of $\mathbb{G}_m$.  Thus it is pulled back from the universal deformation $\Spf(\W(k)[[t]])$, which is $\mathbb{G}_m$ over $\W(k)$.  That is, there is a ring homomorphism $f: \W(k) \to {E_n}_*$ so that $F_n = f^*F$, where $F$ is the universal deformation of $\mathbb{G}_m$ over $\W(k)$.  We may take $ \tx$ to be $f^*(t)$.  Then, since the desired equation\footnote{This is the very definition of the multiplicative group!} holds for $1+t$ in $\W(k)[[t]]$, it also holds in ${E_n}_*[[\tx]]$.

\end{proof}

We therefore obtain an ${E_n}_*$-algebra homomorphism $\alpha: {E_n}_*K(\Z, n+1) \to {E_n}_*$.

\subsection{Obstruction theory}

As in the $A_\infty$ case, there is an adjunction \cite{abghr}
$$\Map_{\rm spectra}(\Sigma^{n+1} H\Z, gl_1 {E_n}) \simeq \Map_{E_\infty}(\Sigma^\infty K(\Z, n+1)_+, {E_n}),$$
and the homotopy of the latter is computable through Andr\'e-Quillen cohomology. We summarize the relevant techniques of Goerss-Hopkins \cite{gh} in the following

\begin{thm}{[Goerss-Hopkins]} \label{gh_thm}
Let $X$, $Y$, and $E$ be $E_\infty$-ring spectra, where $Y$ is an $E$-algebra, and $E_*X$ is flat over $E_*$.  Assume that $E_*$ is a complete local ring with maximal ideal $\m$ and residue field $k := E_* / \m$ of characteristic $p$.  Then, if the cotangent complex $L_f$ for the map
$$f: k \to k \otimes_{E_*} E_*(X)$$
is contractible ($L_f \simeq 0$), the Hurewicz map
$$\Map_{E_\infty}(X, Y) \to \Hom_{\textit{$E_{*}$-alg}}(E_*(X), Y_*)$$
is a homotopy equivalence.

\end{thm}

\begin{proof}

This is assembled from various arguments in sections 4, 6, and 7 of \cite{gh}.  Applying Theorem 4.5, we see that the obstructions to the Hurewicz map being a homotopy equivalence lie in the generalized Andr\'{e}-Quillen cohomology groups $D^s_{E_*T}(E_*X, \Omega^t Y_*)$, where $t>0$, and $T$ is a simplicial $E_\infty$ operad.  Since $E_*$ is a complete local ring, the assumptions of Proposition 6.8 hold to give a spectral sequence
\beqn \label{main_ss_eqn} E_1^{p, q} = D^p_{\EE_k}(k \otimes_{E_*} E_*(X), \Omega^t (\m^q Y_* / \m^{q+1} Y_*)) \implies D^p_{E_*T}(E_*X, \Omega^t Y_*), \eeqn
where $\EE_k$ is an $E_\infty$ operad over $k$.

We aim to show that the abutment of this spectral sequence vanishes, so it suffices to show that the $E_1$-term does.  That $E_1$-term is computable through sequential applications of the spectral sequences of Propositions 6.4 and 6.5:
\beqn \label{ss_eqn} \begin{array}{ccc}
D^p_{\RR}(\Gamma, M)^q \implies D^{p+q}_{\EE_k}(\Gamma, M), & {\rm and} & \Ext^p_{\UU(\Gamma)}(D_q(\Gamma), M) \implies D^{p+q}_{\RR}(\Gamma, M).
\end{array} \eeqn
Here $\RR$ is the Dyer-Lashof algebra, and $D_\RR$ the derived functor of derivations of an unstable algebra over $\RR$.  Similarly, $\UU(\Gamma)$ is the category of unstable modules over an unstable $\RR$-algebra $\Gamma$.  We will apply this in the case $\Gamma = k \otimes_{E_*} E_*(X)$.  

Then the $E_1$-term of \ref{main_ss_eqn} vanishes if
$$D_*(k \otimes_{E_*} E_*(X)) = \pi_*(L_f)$$
(the input to the second spectral sequence in \eqref{ss_eqn}) vanishes.

\end{proof}

We will apply this theorem to the following setting: $E$ is $E_n$, $X = \Sigma^\infty K(\Z, n+1)_+$, and $Y = E_n$.  Then $k \otimes_{{E_n}_*} {E_n}_*(X) = {K_n}_*(X)$; this is isomorphic to the ring 
\beqn \label{reduction_eqn}
K_{n *}(X) \cong \bigotimes_{j\geq 0} R^+_{K_n}(c_j)  = k[c_0, c_1, \dots][ u^{{\pm 1}}] / (c_j^p - u^{p^j(p^n-1)} c_j),
\eeqn

\begin{prop} \label{cotangent_prop}

The cotangent complex for $k \to k \otimes_{{E_n}_*} {E_n}_*K(\Z, n+1)$ is contractible.  Thus the Hurewicz map
$$\Map_{E_\infty}(\Sigma^\infty K(\Z, n+1)_+, E_n) \to \Hom_{\text{${E_n}_*$-alg}}({E_n}_* K(\Z, n+1), {E_n}_*)$$
is a homotopy equivalence.

\end{prop}

\begin{proof}

A nice criterion for the existence of a nullhomotopy of the cotangent complex $L_{B/A}$ is given in Corollary 21.3 of Rezk's notes \cite{rezk_notes}.  Specifically, if the Frobenius maps $\sigma_A: A \to A$ and $\sigma_B: B \to B$ are isomorphisms, then $L_{B/A}$ is contractible.  

Applying this to $A = k$ is immediate, since $k$ is perfect.  For $B = k \otimes_{{E_n}_*} {E_n}_*(X)$, this is a short computation from equation (\ref{reduction_eqn}), using the fact that $c_j^p = u^{p^j (p^n-1)} c_j$.  So the cotangent complex $L_{B/A}$ is contractible.  

\end{proof}

Putting Theorem \ref{gh_thm} and Proposition \ref{cotangent_prop} together, we see that there exists a (unique up to homotopy) $E_\infty$-map $\Phi_n: \Sigma^\infty K(\Z, n+1)_+ \to {E_n}$ which induces $\alpha$ in ${E_n}_*$.  Further, its $p$-adic powers produce a copy of $\Z_p$ inside 
$$[H\Z, gl_1 E_n] \cong \Hom_{\text{${E_n}_*$-alg}}({E_n}_* K(\Z, n+1), {E_n}_*).$$
Now, $\Hom_{\text{${E_n}_*$-alg}}({E_n}_* K(\Z, n+1), {E_n}_*)$ is the Cartier dual of the formal group $\Spf E_n^* K(\Z, n+1)$.  We have seen that this is in fact the multiplicative formal group over ${E_n}_*$, so its Cartier dual is in fact isomorphic to $\Z_p$, yielding Theorem \ref{e(n)_thm}.  Furthermore, reduction from ${E_n}_*$ to ${K_n}_*$ gives an isomorphism from this group to 
$$\Hom_{\text{${K_n}_*$-alg}}({K_n}_* K(\Z, n+1), {K_n}_*),$$
which we have already seen is isomorphic to $\Z_p$.  This gives Proposition \ref{nontriv_prop}.

\begin{defn}

Let $\varphi_n: \Sigma^{n+1} H\Z \to gl_1 {E_n}$ be the map of spectra adjoint to $\Phi_n$.

\end{defn}

We define twisted $E_n$-theory in the same fashion as for $K(n)$.  We first note that there are no nontrivial twistings of ${E_n}$ by $K(\Z, m)$ for $m>n+2$ by Theorem \ref{vanish_thm}.  In contrast for $m=n+2$, the map $\varphi_n$ is essential, since its adjoint induces the nonzero map $\alpha$ in homotopy.  The infinite loop map of its suspension is of the form
$$B (\Omega^\infty \varphi_n): K(\Z, n+2) \to BGL_1 {E_n}.$$
This allows us to define, for every $H \in H^{n+2}(X; \Z)$, the twisted Morava E-theory ${E_n}^*(X; H)$, as in section \ref{sec twistings}.  Its reduction mod $\m$ is the twisted Morava K-theory of Definition \ref{univ_defn}.  We will explore this theory in greater detail in future work.  For instance, it will satisfy analogues of the properties indicated in Theorem \ref{thm prop}.  In the next section, we will consider a truncation that lies between ${E_n}^*(X; H)$ and $K(n)^*(X; H)$.

\begin{rem}

Exactly the same arguments go through to produce a twisting of the completed Johnson-Wilson spectrum $\widehat{E(n)}$ by $H^{n+2}$ when $p=2$.  For odd primes, there is again the matter of the sparsity in the homotopy of $\widehat{E(n)}$ which can be resolved by shortening its period.  Further, when $n$ is even, one must adjoin the root of unity $\xi$ to $\widehat{E(n)}$.  This construction will be used in the next section.

\end{rem}

\section{Twisted integral Morava K-theory}
\label{sec int}

For each prime $p$, there is an integral lift of Morava K-theory which is a spectrum $\tK(n)$ with
$$\pi_*(\tK(n)) = \Z_p[v_n, v_n^{-1}].$$
This is constructed as a quotient of $\widehat{E(n)}$ by the ideal $(v_1, \dots, v_{n-1})$; see, e.g., \cite{kriz, KS1}.

This theory more closely resembles complex K-theory than was the case for the mod $p$ versions (for $n=1$, it is the $p$-completion of K-theory). Therefore, the integral version might be more important for applications
 (indeed, see section \ref{sec-phys}).

\subsection{Properties and twists}

The quotient by $(v_1, \dots, v_{n-1})$ gives a natural transformation $\pi: \widehat{E(n)} \to \tK(n)$.  As an $A_\infty$-map, this gives a well-defined essential twisting
$$\xymatrix@1{K(\Z, n+2) \ar[rr]^-{B (\Omega^\infty \varphi_n)} & & BGL_1 \widehat{E(n)} \ar[rr]^-{BGL_1(\pi)} & & BGL_1 \tK(n)}.$$
In order for the first map to exist, we must assume that $p=2$, or $n$ is odd; if not, we must use instead the quadratic extensions of $\widehat{E(n)}$ and $\tK(n)$ defined by adjoining $\xi$ to their homotopy.  While this twisting exists in these more general cases, we will take $p=2$ for the rest of this paper.

We will continue to refer to the map defined by the above diagram as the universal twisting. It defines a twisted cohomology theory $\tK(n)^*(X; H)$, associated to pairs $(X, H)$, where $X$ is a topological space and $H \in H^{n+2}(X, \Z)$ is a cohomology class.

\begin{rem}
Twisted integral Morava K-theory satisfies properties analogous to those of the mod $2$ version.
In particular, the integral version satisfies the same properties in Theorem \ref{thm prop}.
\end{rem}

The following is almost immediate from Theorem \ref{ahss_thm}:

\begin{thm} \label{int2_thm}

There is a twisted Atiyah-Hirzebruch spectral sequence converging to $\tK(n)^*(X; H)$ with $E_2^{p, q} = H^p(X, \tK(n)^q)$.  The first 
possible nontrivial differential is $d_{2^{n+1}-1}$; this is given by
$$d_{2^{n+1}-1}(xv_n^k) = (\tQ_n(x) + (-1)^{|x|}  x\cup (\tQ_{n-1} \cdots \tQ_1(H)))v_n^{k-1}.$$

\end{thm}

Here $\tQ_k: H^*(X; \Z) \to H^{*+2^{k+1}-1}(X; \Z)$ is an integral cohomology operation lifting the Milnor primitive $Q_k$.

\subsection{Relation with other twisted cohomology theories}
In this section we relate twisted Morava E-theory $E_n^*(X; H)$ and 
twisted integral Morava K-theory $\tK(n)^*(X; H)$
to other, perhaps more familiar, cohomology theories.
We start with the latter.

First, note that the familiar twisted K-theory, constructed by the map $K(\Z, 3) \to BGL_1 K$ may be 2-completed.  That is, 
 the completion map $K \to K_2^{\wedge}$ (where $K_2^{\wedge}$ is the 2-completion of the K-theory spectrum) gives a map $BGL_1 K \to BGL_1 K_2^{\wedge}$.  The composite allows one to define $2$-completed twisted $K$-theory, $(K_2^{\wedge})^*(X; H)$, for $H \in H^3(X; \Z)$.

It is easy to see that the twisted $\tK(1)$-homology, is precisely the same theory:

\begin{prop}

For any space $X$ and class $H \in H^3(X; \Z)$, there is a natural isomorphism
$$\tK(1)^*(X; H) \cong (K_2^{\wedge})^*(X; H).$$

\end{prop}

\begin{proof}

It is immediate that there is an equivalence $f: {K_2^{\wedge}} \to \tK(1)$.  The result follows if we show that the composite
$$K(\Z, 3) \to BGL_1({K_2^{\wedge}}) \to BGL_1(\tK(1))$$
is the universal twist, where the first map is the determinantal twist, and the second is $BGL_1(f)$.  By definition, it is some power of the universal twist.  It follows from the fact that the differentials in the twisted AHSS for $\tK(1)$ are the same as for $K$ (and hence ${K_2^{\wedge}}$) that that power is $1$.

\end{proof}

Next we consider height 2. 
Ando-Blumberg-Gepner have shown \cite{abg} that the theory of topological modular forms $tmf$ admits a twisting by a class in $H^4(X; \Z)$.  That is, there is an essential map $K(\Z, 4) \to BGL_1 tmf$.  Now, $E_2$ is an elliptic spectrum, and so there is a natural transformation $tmf \to E_2$. Consequently, there is an composite map
$$K(\Z, 4) \to BGL_1 tmf \to BGL_1 E_2.$$
As above, this must be some power of the universal twist of $E_2$.  We suspect that, as above,
 the power is a unit.

\subsection{Applications and examples from physics}
\label{sec-phys}

In this section we provide applications to string theory and M-theory. These 
can also be viewed as the motivation for the constructions in this paper.
Due to the relatively low dimensions involved, the cohomology theories will
arise in low degrees.  

\begin{notation} 

We will denote by $\SqZ^{2k+1}$ the integral lift $\SqZ^{2k+1} = \beta \circ Sq^{2k}$.  This is justified by the Adem relation $Sq^1 Sq^{2k} = Sq^{2k+1}$.  It will either be regarded as an operation from mod 2 cohomology to integral cohomology, or from integral cohomology to itself, by precomposing with reduction mod 2.  In contrast, we will continue to write $\widetilde{Q}_i$ for the integral lift of the Milnor primitive.

\end{notation}

\vspace{3mm}
Motivated by structures in string theory, one of the authors conjectured 
in \cite{S1, S2} that the first nontrivial 
differential for the AHSS for (the then also conjectured) 
twisted second integral Morava K-theory at the prime $p=2$
is of the form 
$d_7= \widetilde{Q}_2 + H_7$. 
\footnote{In this section we will specify the degree of the cohomology class $H$ by a subscript.}
The untwisted differential, given by the Milnor primitive 
$\widetilde{Q}_2$, 
is shown  in \cite{KS1} to vanish precisely when the seventh integral 
Stiefel-Whitney class $W_7$ is $0$.  This gives a necessary and sufficient condition for manifolds of dimension at most 12 to be $\tK(2)$-orientable. Thus, the effect of the twist on the differential is 
given by  $H_7 \in H^7(X; \Z)$ coming from a degree seven field on spacetime $X$.

\vspace{3mm}
Indeed, specializing theorem \ref{int2_thm} for the case $n=2$, we get
the differential 
\beqn
d_7(xv_2^k)=(\tQ_2(x) + (-1)^{|x|}  (x\cup \tQ_{1}(H))v_n^{k-1}\;.
\label{d7 twist}
\eeqn
Since $\tQ_1=\SqZ^3$ then we have 
that the vanishing of primary differential \eqref{d7 twist} is equivalent to vanishing of 
the characterstic class $W_7 + \SqZ^3 H_4$. Structures defined by the vanishing of classes 
$W_7 + \alpha_7$, where $\alpha_7$ is a degree seven integral class, are 
introduced in \cite{S3}, where they were called {\it twisted String${}^{K(\Z,3)}$
structures}.  
Therefore, we have 

\begin{prop}
$\widetilde{K}(2)(-, H_4)$, the twisted second integral Morava K-theory at the prime $p=2$ is oriented with respect to the twisted String${}^{K(\Z,3)}$-structure, given by the condition
\beqn \label{w_7_eq}
W_7 + \SqZ^3 H_4=0.
\eeqn
\label{prop string}
\end{prop}
That is, a manifold $X$ of dimension at most 12 with twist $H_4$ satisfying $(\ref{w_7_eq})$ has an orientation class in $\widetilde{K}(2)(-, H_4)$, and thus satisfies Poincar\'{e} duality
in that theory. 

We have seen that the first differential in the Atiyah-Hirzebruch spectral 
sequence in twisted $n$th integral Morava K-theory $\widetilde{K}(n)$ at
the prime $p=2$ is $d_{2^{n+1}-1}$ and is given by the integral lift of the 
Milnor primitive twisted by an integral class of the same degree. 
We now show that for low degrees, namely $n=2$ and $n=3$, the
vanishing of these differentials follows from twisted String 
structures and twisted Fivebrane structures, respectively.

\vspace{3mm}
A twisted String structure is defined by the obstruction class \cite{Wa, SSS3} 
$\frac{1}{2}p_1 + \alpha_4$ being zero, where $\alpha_4$ is a degree four integral class. 
Applying $\SqZ^3=\beta Sq^2$ to this expression gives
$W_7 + \SqZ^3 \alpha_4$. The vanishing of this expression is the vanishing of the
first differential in the Atiyah-Hirzebruch spectral sequence in twisted second integral 
Morava K-theory at the prime $p=2$ with the twist given by the integral class 
$\alpha_4$. Note that the first Spin characteristic class
$P_1=\frac{1}{2}p_1$ pulled back via the classifying map from $H^4(B{\rm Spin}; \Z)$
satisfies $P_1= w_4$ mod 2. 

\vspace{3mm}
Next we consider a twisted Fivebrane structure \cite{SSS3}. The obstruction class
in this case is $\frac{1}{6}p_2 + \alpha_8$, where $\alpha_8$ is a degree eight integral class.
Since we are working at the prime $p=2$, we can take the first factor to be 
$\frac{1}{2}p_2$, that is the second Spin characteristic class $P_2$ pulled back via the
classifying map from $H^8(B{\rm Spin}; \Z)$. 
Note that $P_2$ satisfies $P_2= w_8$ mod 2.

Applying $\SqZ^7=\beta Sq^6 = \SqZ^3 Sq^4$ to
the above obstruction gives $W_{15} + \SqZ^7\alpha_8$.
This can be seen as follows. Assuming the Spin condition,
the Wu formula gives 
$Sq^7 w_8 = w_7 w_8 + w_6 w_9 + w_5 w_{10} + w_4 w_{11} + w_{15}$
Further imposing the condition $w_4=0$ kills $w_5$ (since our manifolds
are oriented), $w_6=Sq^2 w_4 + w_2 w_4$ (by the Spin condition),
and $w_7=Sq^3 w_4=0$. At the integral level, we assume the 
String condition $P_1=\lambda=0$, which is a natural condition to impose
since we are dealing with Fivebrane structures. 

Now let us make the assumption that there is an integral class $H_5$ such that $\SqZ^3 H_5 = \alpha_8$.  A computation gives $\tQ_2 \tQ_1 H_5 = \SqZ^7 \SqZ^3 H_5 = \SqZ^7 \alpha_8$.  Therefore, if we twist the third integral Morava K-theory of this space by $H_5$,  we can identify $W_{15} + \SqZ^7 \alpha_8$ as the first differential in the tAHSS.  Thus the twisted Fivebrane structure ensures the vanishing of the first differential.

We summarize what we have in the following 

\begin{prop} 
A twisted Fivebrane structure of the form $W_{15} + \SqZ^7 (\SqZ^3 H_5)$ on a String manifold implies the vanishing of the first obstruction to 
orientation with respect to twisted third integral Morava K-theory
$\widetilde{K}(3)(-, H_{5})$ 
at the prime $p=2$. 
\end{prop}

Similarly, and as a special case of proposition \ref{prop string},
 a twisted String structure implies the vanishing of the first obstruction to orientation
with respect to twisted second integral Morava K-theory $\widetilde{K}(2)(-, H_4)$.

\vspace{2mm}
We now provide two concrete examples from physics.

\paragraph{Example 1: Classes in heterotic string theory.}
Heterotic string theory is defined on two disconnected 
components of 10-dimensional spacetime $X^{10}$ 
each with an $E_8$ vector bundle. Let $a$ and $b$ be the
degree four characteristic classes characterizing these 
bundles, and let $\lambda=P_1$ be the first Spin characteristic class
$\frac{1}{2}p_1$ of $X^{10}$. Anomaly cancellation imposes
the linear relation $a + b = \lambda$ (see \cite{DMW}). 
The classes $a$ and $b$  can a priori take any integral value,
and can be torsion. 
Now we assume, in the spirit of \cite{DMW},
 that one of the classes $a$ and $b$ 
can be lifted to K-theory. Hence, let us take $a$ to satisfy 
$\SqZ^3 a=0$. From the linear relation above, we can see
that this condition translates into a corresponding condition
on $\lambda -b$, namely $\SqZ^3(\lambda - b)=0$. 
Since the Steenrod operation is linear and $\SqZ^3\lambda=W_7$,
we get from this that 
\[
W_7(X^{10}) + \SqZ^3 b=0\;.
\]
Therefore, we have the following 
\begin{thm}
Consider heterotic string theory on 10-dimensional Spin manifold $X^{10}$
and with $E_8$ bundles on the two boundary components characterized by 
the degree four characteristic classes $a$ and $b$. 
Assuming that $a$ satisfies $Sq^3 a=0$, the anomaly cancellation requires 
$X^{10}$ to be oriented with respect to twisted second integral Morava K-theory 
$\widetilde{K}(2)(-, b)$ at the prime $p=2$, where the twist $b$ manifests itself through the class $H_7 = Sq^3 b$ in the AHSS.
\label{thm het}
\end{thm}
One justification for the assumption in the theorem is that the class
$a$ involves the Chern character in cohomology. The breaking of 
symmetry between $a$ and $b$ is possible because each is defined on 
a separate component of spacetime, and this 
 is common in the process 
dimensional reduction whose aim is the search for realistic models in four dimensions.

\vspace{3mm}
In the formulation of the partition function of the C-field in 
M-theory, one encounters an anomaly given by the seventh integral 
Stiefel-Whitney class $W_7$ of spacetime \cite{DMW}.
This anomaly is cancelled in \cite{KS1} by insisting that spacetime 
be oriented with respect to integral Morava K-theory $\widetilde{K}(2)$ at the prime 
$p=2$.  Thus theorem \ref{thm het}
 generalizes the corresponding result in \cite{KS1} to the twisted case (although
 in a slightly different, but related, setting).

\paragraph{Orientation with respect to twisted relative Morava K-theory $\widetilde{K}(2)$.}
The relation of the first differential to the Stiefel-Whitney classes 
in the untwisted case is given in \cite{KS1}.  We now work out the analog in the 
relative twisted case. As in \cite{KS1} we focus on manifolds $X$ of dimension 
at most 12, as appropriate for
string theory. The integral lift $\widetilde{Q}_2$ of the Milnor primitive 
is given by $\beta Sq^6$ plus elements of lower Cartan-Serre filtration. 
Here the Steenrod square is the relative cohomology operation 
$Sq^i: H^m (X, A; \Z / 2) \to H^{m+i}(X, A; \Z / 2)$ and $\beta$ is the relative
Bockstein. The main application occurs 
 when $X$  is 12-dimensional and $A=\partial X$ is the 
11-dimensional boundary.

We have

\begin{prop}
For a relative pair $(X, A)$ with $X$ of dimension $\leq 12$,
the orientation condition with respect to twisted relative integral
Morava K-theory at the prime 2, $\widetilde{K}(2)(X, A; H_4)$, 
is $W_7(X, A) + Sq^3 H_4=0$.
\end{prop}

\paragraph{Applications of twisted Morava E-theory.}
To consider twisted Morva E-theory, we essentially work with the above examples, 
 requiring in this case that our spacetime be also Spin.
In \cite{KS1} it was shown that under this condition, the 
orientation condition in Morava $E_2$-theory at the prime $p=2$ 
is  given by $W_7=0$. That is, the Spin assumption ensures that $d_3 = 0$, and so the
first differential is $d_7$, given by the integral lift of the Milnor primitive $\widetilde{Q}_2$.
In the twisted Morava E(2)-theory constructed above, we will have an 
orientation condition $W_7 + \SqZ^3 H_4=0$, as in the case for $\tK(2)(-, H_4)$.

\begin{thm}
A Spin spacetime is oriented with respect to twisted Morava $E_2$-theory at the 
prime $p=2$ if $W_7 +  \SqZ^3 H_4=0$. 
\end{thm}

Instead of repeating the above examples for the Spin case, which amounts to 
replacing Morava K-theory
with Morava E-theory, we will provide another
application.

\paragraph{Example 2. The partition function in M-theory on a circle.}
We will consider the setting of \cite{DMW}. 
M-theory on a Spin 11-dimensional manifold $Y^{11}$ has a degree four 
field whose characteristic class $G$ 
is essentially the  class $a$ of an $E_8$ bundle over
$Y^{11}$. Taking $Y^{11}=X^{10} \times S^1$, one would like 
to relate $a$ to classes on $X^{10}$ that admit a K-theoretic 
description. This places constraints on $G$. 
The theory is characterized by the partition function, ideally a complex
number with phase $\phi_a=(-1)^{f(a)}$, where $f(a)$ is the mod 2 index of 
the Dirac operator coupled to the $E_8$ bundle. This is a topological
invariant in ten dimensions which is not quite additive but rather satisfies a quadratic 
refinement
$$f(a + a')= f(a) + f(a') + \langle a \cup Sq^2 a', [X^{10}] \rangle$$
for $a, a'\in H^4(X^{10}; \Z)$. Evaluation 
of the partition function requires dealing with the torsion pairing 
$T: H^4_{\rm tors}(X^{10}; \Z) \times H^7_{\rm tors}(X^{10}; \Z) \to \Z/2$
defined by $T(a, b)= \langle a \cup c, [X^{10}] \rangle$, where 
$\beta (c)=b$, and is a Pontrjagin duality between 
$H^4_{\rm tors}$ and $H^7_{\rm tors}$. 

In \cite{DMW} a variation under $a \mapsto a + 2b$, for $b$ torsion, 
was considered; this leads to the condition $W_7(X^{10})=0$. 
We will instead consider the variation under $a \mapsto a + 3b$ for $b$
satisfying $f(b)=0$. This does not necessarily imply that 
$f(2b)=0$ because of the cross-term coming from the quadratic refinement. 
We then have 
$$
f(a + 3b)= f(a) + f(2b) + \int_{X^{10}} b \cup Sq^2 a\;.
$$

Now, using \cite{stong}, or equation (3.23) in \cite{DMW}, we have $f(2b)=\int_{X^{10}} b \cup Sq^2 \lambda$, where $\lambda= \frac{1}{2}p_1(Y^{11})$.
The phase of the partition function is then invariant under the variation 
of $a$ by $2b$ if
$\int_{X^{10}} b \cup Sq^2 \lambda + \int_{X^{10}} b \cup Sq^2 a=0$.
Since $b$ is torsion, this is satisfied if $Sq^3 \lambda + Sq^3 a=0$, and so 
$W_7(X^{10}) + \SqZ^3 a=0$. This is the first differential in 
twisted Morava $E_2$-theory.
Therefore, we have

\begin{prop}
The partition function is invariant (in the above sense) if spacetime $X^{10}$ is 
oriented with respect to twisted Morava $E_2$-theory, with the twist given by 
$a$, the class of the $E_8$ bundle over $X^{10}$. 
\end{prop}

This is an extension to the twisted case of the 
result in \cite{KS1}.

\vspace{3mm}
We now give another occurrence of this condition in the above setting.
The  variation $a \mapsto a + 2b$ for $b$ torsion leads to $W_7(X^{10})=0$
\cite{DMW}. Then taking $c \in L=H^4_{\rm tors}/2H^4_{\rm tors}$, a vector space
over $\Z / 2$, the function $f(a +c)$ is linear in $c$ on the 
vector space $L'=\{ c \in L ~|~ Sq^3 a=0\}$. Then Pontrjagin duality 
of the torsion pairing $T$ implies that there exists $P \in H^7(X^{10}; \Z)$
with $f(c)=T(c, P)$ for all $c$. Then 
we have $f(a + c)= f(a)  + T(c, Sq^3 a + P)$. Again, requiring invariance of the phase
imposes the condition $Sq^3a =P$ mod $Sq^3(H^4_{\rm tors})$. 
Now, if we impose the condition
$a= \lambda$
 on the class $a$, then we have 
$W_7(X^{10})=Sq^3a$.
Therefore, again 

\begin{prop}
Consistency (in the above sense) of the M-theory partition function 
requires orientation with respect to twisted $E_2$-theory, with the twist 
again given by the class of the $E_8$  bundle. 
\end{prop}

\end{document}